\documentclass[a4paper,10.5pt]{article}
\usepackage{fullpage}
\usepackage{listings}
\usepackage{latexsym, amsfonts, amssymb, amsmath, amsthm, mathrsfs, mathtools, setspace, graphics, graphicx, bbm, float,bigints, calc}
\usepackage{enumerate}
\usepackage{caption}
\usepackage{indentfirst}
\usepackage{framed}
\usepackage{yhmath}
\usepackage{bm}

\usepackage{color}
\usepackage[dvipsnames]{xcolor}
\usepackage{pdfpages}

\usepackage{lineno}

\usepackage{geometry}
\RequirePackage{tikz}
\usetikzlibrary{calc,trees,positioning,arrows,chains,shapes.geometric,%
    decorations.pathreplacing,decorations.pathmorphing,shapes,%
    matrix,shapes.symbols}
\allowdisplaybreaks    

\usepackage[affil-it]{authblk}

\usepackage{hyperref}

\newtheorem{theorem}{Theorem}[section]

\newtheorem{lemma}[theorem]{Lemma}

\newtheorem{coro}[theorem]{Corollary}
\newtheorem{defi}[theorem]{Definition}

\newtheorem{prop}[theorem]{Proposition}
\newtheorem{remark}[theorem]{Remark}

\numberwithin{equation}{section}

\newcommand{\vertiii}[1]{{\left\vert\kern-0.25ex\left\vert\kern-0.25ex\left\vert #1 
    \right\vert\kern-0.25ex\right\vert\kern-0.25ex\right\vert}}

\newcommand{\PP}{\mathbb{P}}
\newcommand{\E}{\mathbb{E}}
\newcommand{\QQ}{\mathbb{Q}}

\newcommand{\RD}{\mathbb{R}^{d}}
\newcommand{\RR}{\mathbb{R}}
\newcommand{\R}{\mathbb{R}}

\newcommand{\EE}{\mathbb{E}\,}

\definecolor{darkgreen}{rgb}{0.0, 0.5, 0.0}

\title{A statistical approach for simulating the density solution of a McKean-Vlasov equation}
\author[1]{Marc Hoffmann}
\author[2]{Yating Liu}
\affil[1]{ CEREMADE, CNRS, Universit\'e Paris-Dauphine, PSL,  and Institut Universitaire de France,  75016 Paris, France}
\affil[2]{ CEREMADE, CNRS, Universit\'e Paris-Dauphine, PSL,  75016 Paris, France}

{   \makeatletter
    \renewcommand\AB@affilsepx{: \protect\Affilfont}
    \makeatother

    \makeatletter
    \renewcommand\AB@affilsepx{, \protect\Affilfont}
    \makeatother

    \affil[1]{hoffmann@ceremade.dauphine.fr}
    \affil[2]{liu@ceremade.dauphine.fr}}

\newcommand{\Cmodel}{\mathfrak{M}}

\newtheorem{assump}{Assumption}

\begin{document}
\maketitle
\begin{abstract}
We prove optimal convergence results of a stochastic particle method for computing the classical solution of a multivariate McKean-Vlasov equation, when the measure variable is in the drift, following the classical approach of \cite{bossy1997stochastic, MR1910635}. Our method builds upon adaptive nonparametric results in statistics that enable us to obtain a data-driven selection of the smoothing parameter in a kernel type estimator. In particular, we generalise the Bernstein inequality of \cite{della2021nonparametric} for mean-field McKean-Vlasov models to interacting particles Euler schemes and obtain sharp deviation inequalities for the estimated classical solution. We complete our theoretical results with a systematic numerical study, and gather empirical evidence of the benefit of using high-order kernels and data-driven smoothing parameters.  
\end{abstract}

\section{Introduction}

\subsection{Motivation} \label{sec: motivation}

Let $\mathcal P_1$ denote the set of probability distributions on $\R^d$, $d \geq 1$, with at least one moment. Given a time horizon $T>0$, an initial condition $\mu_0 \in \mathcal P_1$ and two functions $$b:[0,T]\times \R^d \times \mathcal P_1 \rightarrow \R^d\;\;\text{and}\;\;\sigma:[0,T] \times \R^d \rightarrow \R^d \otimes \R^d,$$
a probability solution $\mu = (\mu_t(dx))_{t \in [0,T]}$ of the nonlinear Fokker-Planck equation
\begin{align}\label{original}
\begin{cases}
\partial_t \mu +\mathrm{div}\big(b(\cdot,\cdot,\mu)\mu\big) = \tfrac{1}{2}\sum_{k,k'=1}^d \partial_{kk'}^2\big(\sigma \sigma^\top \mu\big)\\
\mu_{t = 0}= \mu_0
\end{cases}
\end{align}
exists, under suitable regularity assumptions on the data $(\mu_0,b,\sigma)$, see in particular Assumptions \ref{ass: AI}, \ref{ass: AII} and (\ref{ass: AIII 1} or \ref{ass: AIII 2}) in Section \ref{sec: assumptions} and \ref{sec: nonlinear} below. Moreover, for each $t>0$, $\mu_t(dx)$ has a smooth density, also denoted by $\mu_t(x)$, which is a  classical solution to \eqref{original}. There is a vast literature on approximating \eqref{original}, we mention \cite{dos2022simulation, kumar2022well, leobacher2022well, reisinger2022adaptive, szpruch2021antithetic} to name but a few. The objective of the paper is to construct a probabilistic numerical method to compute $\mu_T(x)$ for every $T>0$ and $x \in \R^d$, with accurate statistical guarantees, that improves on previous works on the topic using an interacting particle system simulated from an Euler scheme. The main novelty is the introduction of an adaptive nonparametric statistical methodology, that enables us to obtain confidence intervals for the classical solution of  \eqref{original}, and this departs from what is usually sought when controlling the classical weak or strong error in an Euler scheme.\\

Model \eqref{original} includes in particular nonlinear transport terms of the form $b(t,x,\mu) = F\big(f(x)+\int_{\R^d}g(x-y)\mu(dy)\big)$, for smooth functions $F,f,g:\R^d \rightarrow \R^d$ that account for evolution systems with common force $f$ and mean-field interaction $g$ under some nonlinear transformation $F$ together with diffusion $\sigma$. Such models and their generalizations have inspired a myriad of application domains over the last decades, ranging from physics \cite{martzel2001mean, bossy2021instantaneous} to neurosciences \cite{MR2974499, bossy2015clarification}, structured models in population dynamics (like in {\it e.g.} swarming and flocking models \cite{MR2974499, bossy2015clarification, mogilner99, BURGER}), together with finance and mean-field games \cite{lasry2018mean, cardaliaguet2018mean, MR3752669}, social sciences and opinion dynamics \cite{CHAZELLE}, to cite but a few.\\ 

The classical approach of Bossy and Talay \cite{bossy1997stochastic}, subsequently improved by Antonelli and Kohatsu-Higa \cite{MR1910635} paved the way: the probability $\mu_t$ has a natural interpretation as the distribution $\mathcal L(X_t)$ of an $\R^d$-valued random variable $X_t$, where the stochastic process  $(X_{t})_{t\in[0, T]}$ solves  a McKean-Vlasov equation 
\begin{align} \label{mckeaneq}
\begin{cases}
dX_{t}=b(t, X_{t}, \mathcal L(X_t))dt + \sigma(t, X_{t})dB_{t}\\
\mathcal L(X_0) = \mu_0,
\end{cases}
\end{align}
for some standard $d$-dimensional Brownian motion $B = (B_t)_{t \in [0,T]}$. The mean-field interpretation of \eqref{mckeaneq} is the limit as $N \rightarrow \infty$ of an interacting particle system
\begin{align*}
\begin{cases}
dX_{t}^n=b(t, X_{t}^n, N^{-1}\sum_{m = 1}^N\delta_{X_t^m})dt + \sigma(t, X_{t}^n)dB_{t}^n,\;\;1 \leq n \leq N,\\
\mathcal L(X_0^1,\ldots, X_0^N) = \mu_0^{\otimes N},
\end{cases}
\end{align*}
where the $(B_t^n)_{t \in [0,T]}$ are independent Brownian motions, and $\delta_a$ denotes the Dirac measure at a point $a\in\RD$.  Indeed, under fairly general assumptions, see {\it e.g.} \cite{mckean1967propagation, sznitman1991topics, meleard1996asymptotic} we have $N^{-1}\sum_{n = 1}^N\delta_{X_t^n} \rightarrow \mu_t$ weakly as $N \rightarrow \infty$. One then simulates an approximation of 
\begin{equation} \label{eq: interacting system}
(X_t^1,\ldots, X_t^N)_{t \in [0,T]}
\end{equation}
by a stochastic system of particles 
\begin{equation} \label{eq: synthesised}
(\bar X_t^{1,h},\ldots, \bar X_t^{N,h})_{t \in [0,T]}
\end{equation}
using an Euler scheme with mesh $h$, see \eqref{discresys} below for a rigorous definition. The synthesised data \eqref{eq: synthesised} are then considered as a proxy of the system \eqref{eq: interacting system} and a simulation of $\mu_T(x)$ is obtained via the following nonparametric kernel estimator  
\begin{equation} \label{eq: def kernel estimator}
\widehat \mu_T^{N,h,\eta}(x) := N^{-1}\sum_{n = 1}^N \eta^{-d}K\big(\eta^{-1}(x-\bar X_T^{n,h})\big).
\end{equation}
Here $K:\R^d \rightarrow \R$ is kernel function that satisfies $\int_{\R^d}K(x)dx=1$ together with additional moment and localisation conditions, see Definition \ref{def: kernel} below, and $\eta$ is the statistical smoothing parameter.   Note that the computational cost of \eqref{eq: def kernel estimator} with a tensor–product kernel and coordinate-wise bandwidths  (see also Definition \ref{def: kernel} and Remark \ref{rk: construction kernels}) at one point $x$ is of order $O(Nd)$, where $N$ is the number of particles and $d$ the dimension, which is the same as in standard kernel density estimation with i.i.d. samples. \\

When $K$ is a Gaussian kernel, 
\cite{bossy1997stochastic} and \cite{MR1910635}  obtained  some convergence rates that depend on $N, \eta, h$ and the data $(\mu_0,b,\sigma)$. See also the comprehensive review of Bossy \cite{bossy2005some}.
 Our objective is somehow to simplify and improve on these results, up to some limitations of course, by relying on adaptive statistical methods. In order to control the strong error $\EE[|\widehat \mu_T^{N,h,\eta}(x)-\mu_T(x)|^2]$, a crucial step is to select the optimal bandwidth $\eta$ that exactly balances the bias and the variance of the statistical error and that depends on the smoothness of the map $x \mapsto \mu_T(x)$. To do so, we rely on recent data-driven selection procedures for $\eta$ based on the Goldenshluger-Lepski's method in statistics \cite{GL08, GL11, GL14}. They do not need any prior knowledge of $x \mapsto \mu_T(x)$, and this is a major advantage since the exact smoothness of the solution map
$$(\mu_0, b,\sigma) \mapsto (x \mapsto \mu_T(x))$$
may be difficult to obtain, even when $\mu_0$, $b$, $\sigma$ are given analytically. At an abstract level, if we assume that we can observe $(X_t^1,\ldots, X_t^N)_{t \in [0,T]}$ exactly, a statistical theory has been recently proposed to recover $\mu_T(x)$ in \cite{della2021nonparametric}, a key reference for the paper. However, the fact that we need to simulate first a proxy of the idealised data via a system of interacting Euler schemes needs a significant adjustment in order to obtain precise error guarantees. For other recent and alternate approaches  based on  projection-based particle method and the particle regression Monte Carlo
method,  we refer to \cite{belomestny2018projected, belomestny2020optimal} and the references therein.

\subsection{Results and organisation of the paper}

In Section \ref{sec: assumptions} we give the precise conditions on the data $(\mu_0, \sigma, b)$. We need a smooth and  sub-Gaussian integrable initial condition $\mu_0$ in order to guarantee the existence of a classical solution $\mu_t(x)$ to the Fokker-Planck equation \eqref{original} and sub-Gaussianity of the process $(X_t)_{t \in [0,T]}$ that solves the McKean-Vlasov equation \eqref{mckeaneq}. The coefficients $b(t,x,\mu_t),\,\sigma(t,x)$ are smooth functions of $(t,x)$ and $\sigma$ is  uniformly elliptic, following the conditions of Gobet and Labart \cite{Gobet2008sharp}, where the expression “small time” in the title of \cite{Gobet2008sharp} refers to a small time neighborhood rather than to a small value of $T$. \\

In Section \ref{sec: esti} we construct the estimator $\widehat \mu_T^{N,h,\eta}(x)$ of \eqref{eq: def kernel estimator} via a regular kernel $K$ of order $\ell$, see Definition \ref{def: kernel}. It depends on the size $N$ of the particle system, the step $h>0$  of the Euler scheme and the statistical smoothing parameter $\eta$ and the order of the kernel $\ell$. Abbreviating $K_\eta = \eta^{-d}K(\eta^{-1}\cdot)$ for $\eta >0$, we have the following decomposition for the simulation error:
\begin{align*}
\widehat \mu_T^{N,h,\eta}(x)-\mu_T(x) & = \int_{\R^d} K_\eta(x-\cdot) d(\bar \mu^{N,h}_T-\bar \mu_T^h)  + K_\eta \star (\bar \mu_T^h- \mu_T)(x)
+ K_\eta \star \mu_T(x) - \mu_T(x),
\end{align*}
where $\star$ denotes convolution, $\bar \mu^{N,h}_T = N^{-1}\sum_{n = 1}^N\delta_{\bar X_T^{n,h}}$ and $\bar \mu_T^h$ is the probability distribution of the continuous Euler scheme at time $t=T$, constructed  in \eqref{eq: euler abstract} below. In turn, the study of the error splits into three terms, with only the first one being stochastic.  We prove in Proposition \ref{prop: bernstein} a Bernstein concentration inequality for the fluctuation of $\bar \mu^{N,h}_T(dy)-\bar \mu_T^h(dy)$. In contrast to \cite{malrieu2006concentration}, it has the advantage to encompass test functions that may behave badly in Lipschitz norm, while they are stable in $L^1$ like a kernel $K_\eta$ for small values of $\eta$. It reads, for an arbitrary bounded test function $\varphi$:
$$\PP\Big( \int_{\R^d}\varphi \,d(\bar \mu_T^{N,h}-\bar \mu_T^h) \geq \varepsilon \Big) \leq \kappa_1 \exp\Big(-\frac{\kappa_2N\varepsilon^2}{|\varphi|^2_{L^2(\mu_T)}+|\varphi|_\infty\varepsilon}\Big),\;\;\text{for every}\;\;\varepsilon \geq 0,$$
with explicitly computable constants $\kappa_1$ and $\kappa_2$ that only depend on $T$ and the data $(\mu_0,  b, \sigma)$.
The proof is based on a change of probability argument via Girsanov's theorem. It builds on ideas developed in \cite{lacker2018on} and \cite{della2021nonparametric}, but it requires nontrivial improvements in the case of approximating schemes, in particular fine estimates in Wasserstein distance between the distributions $\bar \mu_T^h$ and $\mu_T$, thanks to Liu \cite{liu2019optimal}. When the measure argument in the drift is nonlinear, we impose some smoothness in terms of linear differentiability, see Assumption \ref{ass: AIII 2}. One limitation of this approach is that we can only encompass a measure term in the drift but not in the diffusion coefficient. The second approximation term $K_\eta \star (\bar \mu_T^h(x)-\mu_T(x)\big)$  is managed via the uniform estimates of Gobet and Labart  \cite{Gobet2008sharp} of $\bar \mu_T^h(x)-\mu_T(x)$ so that we can ignore the effect of $K_\eta$ as $\eta \rightarrow 0$ by the $L^1$-stability $|K_\eta|_{L^1} = |K|_{L^1}$. The final bias term  $K_\eta \star \mu_T(x) - \mu_T(x)$ is controlled by standard kernel approximation.\\ 

In Theorem \ref{thm: deviation} we give our main result, namely a deviation probability for the error $\widehat \mu_T^{N,h,\eta}(x)-\mu_T(x)$, provided $h$ and $\eta$ are small enough. In Theorem \ref{thm: error 1}, we give a quantitative bound for the expected error of order $p \geq 1$ that is further optimised in Corollary \ref{cor: error} when $x \mapsto \mu_T(x)$ is $k$-times continuously differentiable. We obtain a (normalised) error of order $N^{-k/(2k+d)}+h$, which combines the optimal estimation error in nonparametric statistics with the optimal strong error of the Euler scheme. Finally, we construct in Theorem \ref{thm: oracle} a variant of the Goldenshluger-Lepski's method that automatically selects an optimal data-driven smoothing parameter $\eta$ without requiring any prior knowledge on the smoothness of the solution. As explained in Section \ref{sec: motivation}, our approach suggests a natural way to take advantage of adaptive statistical methods in numerical simulation, when even qualitative information about the object to simulate  are difficult to obtain analytically.\\   

We numerically implement our method on several examples in Section \ref{sec: simulations}. We show in particular that it seems beneficial to use high-order kernels ({\it i.e.} with many vanishing moments) rather than simply Gaussian ones or more generally even kernels that only have one vanishing moment. This is not a surprise from a statistical point of view, since $x \mapsto \mu_T(x)$ is usually very smooth. Also, the data-driven choice of the bandwidth seems useful even in situations where our assumptions seem to fail (like for the Burgers equation, see in particular Section \ref{sec: burgers}). The proofs are delayed until Section \ref{sec: proofs}.

\section{Main results}

\subsection{Notation and assumptions}  \label{sec: assumptions}

For $x \in \R^d$, $|x|^2 = x \cdot x^\top$ denotes the Euclidean norm. We endow $\mathcal P_1$  with the $1$-Wasserstein metric
$$\mathcal W_1(\mu,\nu) = \inf_{\rho \in \Gamma(\mu,\nu)}\int_{\R^d \times \R^d} |x-y|\rho(dx,dy) = \sup_{|\varphi|_{\mathrm{Lip} \leq 1}}\int_{\R^d}\varphi \, d(\mu-\nu),$$
where $\Gamma(\mu,\nu)$ denotes the set of probability measures on $\R^d \times \R^d$ with marginals $\mu$ and $\nu$, and $|\varphi|_{\mathrm{Lip}} = \sup_{x \neq y}\frac{|\varphi(y)-\varphi(x)|}{|y-x|}$ is the Lipschitz semi-norm of $\varphi$.\\

All $\R^d$-valued functions $f$ defined on $[0,T]$, $\R^d$, $\mathcal P_1$ (or product of those) are implicitly measurable for the Borel $\sigma$-field induced by the (product) topology and written componentwise as $f = (f^1,\ldots, f^d)$, where the $f^i$ are real-valued. Following \cite{Gobet2008sharp}, we say that a function $f:[0,T] \times \R^d \rightarrow \R^d$ is $\mathcal C^{k,l}_b$ for integers $k,l\geq 0$ if the real-valued functions  $f^i$ in the representation $f = (f^1,\ldots, f^d)$ are  all bounded  and continuously differentiable with uniformly bounded derivatives w.r.t. $t$ and $x$ up to order $k$ and $l$ respectively. We write $\mathcal C^k_b$ (or $\mathcal C_b^l$) if $f$ is $\mathcal C^{k,l}_b$ and only depends on the time (or space) variable.\\

 For $\varphi : \R^d \rightarrow \R$, a $\sigma$-finite measure $\nu$ on $\R^d$ and $1 \leq p < \infty$, we set 
$$|\varphi|_{L^p(\nu)} = \Big(\int_{\R^d}|\varphi(x)|^p\nu(dx)\Big)^{1/p},\;\;|\varphi|_{L^p} = \Big(\int_{\R^d}|\varphi(x)|^p dx\Big)^{1/p},\;\;|\varphi|_\infty = \sup_{x \in \R^d}|\varphi(x)|.$$

\begin{assump}\label{ass: AI}
The law $\mathcal L(X_0) = \mu_0(dx)$ in \eqref{mckeaneq} or equivalently in \eqref{original} is sub-Gaussian, that is,  $\int_{\R^d}\exp(\gamma|x|^2)\,\mu_0(dx) < \infty$ for some $\gamma >0$. 
\end{assump}

The strong integrability assumption of $\mu_0$ (together with the subsequent smoothness properties of $b$ and $\sigma$) will guarantee that $\mu_t$ is sub-Gaussian for every $t>0$. This is the gateway to our change of probability argument in the Bernstein inequality of the Euler scheme in Proposition \ref{prop: bernstein} below.

\begin{assump}\label{ass: AII}
The diffusion matrix $\sigma: [0,T] \times \R^d \rightarrow \R^d \otimes \R^d$ is uniformly elliptic, $\mathcal C^{1,3}_b$ and $\partial_t \sigma$ is $\mathcal C_b^{0,1}$.
\end{assump}

By uniform ellipticity, we mean $\inf_{t,x} \xi^\top c(t,x) \xi \geq C_{\sigma} |\xi|^2$ for every $\xi \in \R^d$ and some $C_{\sigma} >0$, where $c = \sigma \sigma^\top$.

As for the drift  $b:[0,T]\times \R^d\times \mathcal P_1 \rightarrow \R^d$, we separate the case whether the dependence in the measure argument is linear or not. A nonlinear dependence is technically more involved and this case is postponed to Section \ref{sec: nonlinear}.

\begin{assump}(Linear representation of the drift.)
\label{ass: AIII 1}
We have
$$b(t,x,\mu) = \int_{\R^d} \widetilde b(t,x,y)\mu(dy),$$
with $(x,y) \mapsto \widetilde b(t,x,y)$ Lipschitz continuous (uniformly in $t$), $x\mapsto \widetilde{b}(t, x, y)\in \mathcal{C}_b^3$ (uniformly in $t$ and in $y$), $y\mapsto \widetilde{b}(t, x, y)\in \mathcal{C}_b^2$ (uniformly in $t$ and in $x$), $(t,x,y) \mapsto \partial_t \widetilde b(t,x,y)$ continuous and bounded and $\sup_{t \in [0,T]}|\widetilde b(t,0,0)|<\infty$. 
\end{assump}

Assumption \ref{ass: AIII 1} implies that the function $(t,x,\mu)\mapsto b(t,x,\mu)$ is Lipschitz continuous,  that is, there exists a constant $|b|_{\text{Lip}}>0$  such that  for every $(t,x,\mu), (s, y, \nu)\in[0,T]\times \RD\times \mathcal{P}_1,$ 
\begin{equation}\label{inq:lipb}
|b(t,x,\mu)-b(s, y, \nu)|\leq |b|_{\text{Lip}}\big[|t-s|+|x-y|+\mathcal{W}_1(\mu, \nu)\big],
\end{equation}
where the conclusion follows from the Kantorovich duality for the Wasserstein distance (see, e.g., \cite[Corollary 5.4]{MR3752669}).  
As detailed in the proof of Proposition \ref{densityeuler} below, Assumption \ref{ass: AIII 1} also implies that $(t,x) \mapsto b(t,x,\mu_t)$ is $\mathcal C^{1,3}_b$. This together with Assumption \ref{ass: AII} on the diffusion coefficient enables us to obtain a sharp approximation of the density of the Euler scheme of \eqref{mckeaneq}, a key intermediate result.  This is possible thanks to the result of Gobet and Labart \cite{Gobet2008sharp}, which, to our knowledge, is the best result in that direction. See also \cite{Bally1996The, MR1931967,Guyon2006Euler} for similar results under stronger assumptions.

Throughout the paper, we use the same symbol $\mu_t$ to denote both the probability distribution and its density (the same convention applies to $\bar\mu^h_t$, see Definition \ref{def:abstract-euler-scheme}). Specifically, $\mu_t(dx)$ denotes the distribution, while $\mu_t(x)$ denotes its density with respect to the Lebesgue measure. This slight abuse of notation will be maintained throughout.  Regarding constants, we denote by $\Cmodel$ generic constants that may vary from line to line and depend only on the model parameters $\mu_0, b, \sigma, T,d$. When the precise dependence is needed, we make it explicit in order to emphasize the links between results.  

\subsection{Construction of an estimator of $\mu_T(x)$} \label{sec: esti}

We pick an integer $M\in\mathbb{N}\setminus \{0\}$ for time discretisation and $h\coloneqq\frac{T}{M}$ as time step. We set $t_m\coloneqq m\cdot h$, $m\in \{0, ..., M\}$.\\

For $1\leq n \leq N$, let $B^n=(B_t^{n})_{t\in[0, T]}$ be $N$ independent $d$-dimensional Brownian motions. We construct $N$ interacting processes $(\bar X_t^{1,h},\ldots, \bar X_t^{N,h})_{t \in [0,T]}$ via the following Euler schemes   
\begin{equation}\label{discresys}
\begin{cases}
\:\bar{X}_{t_{m+1}}^{n, h}=\bar{X}_{t_{m}}^{n, h} + h\cdot b(t_m, \bar{X}_{t_{m}}^{n, h}, \bar{\mu}_{t_{m}}^{N,h}) + \sqrt{h}\cdot \sigma (t_m, \bar{X}_{t_{m}}^{n, h}) Z_{m+1}^{n},\quad 1\leq n\leq N, \\ 
\;Z_{m+1}^{n}\coloneqq \frac{1}{\sqrt{h}}\big( B_{t_{m+1}}^n- B_{t_{m}}^n\big),\;\bar{\mu}_{t_{m}}^{N,h}\coloneqq \frac{1}{N}\sum_{n=1}^{N}\delta_{\bar{X}_{t_{m}}^{n, h}},\\ 
\mathcal L\big(\bar{X}_0^{1, h}, \ldots , \bar{X}_0^{h, N}\big) = \mu_0^{\otimes N},
\end{cases}
\end{equation}
with $\big(\bar{X}_0^{1, h}, \ldots , \bar{X}_0^{h, N}\big)$ independent of $(B^n)_{1 \leq n \leq N}$.  They are appended with their continuous version: for every $0\leq m\leq M-1$, $t\in[t_m, t_{m+1})$ and $1\leq n\leq N$:
\begin{align*}\label{contsys}
 \bar{X}_{t}^{n, h}  =\bar{X}_{t_{m}}^{n, h} + (t-t_m) b(t_m, \bar{X}_{t_{m}}^{n, h}, \bar{\mu}_{t_{m}}^{N, h}) + \sigma (t_m, \bar{X}_{t_{m}}^{n, h}) (B^{n}_t-B^{n}_{t_m}).
\end{align*}

\begin{defi} \label{def: kernel}
Let $\ell \geq 0$ be an integer. A $\ell$-regular kernel is a 
bounded function $K:\R^d \rightarrow \R$ such that  for $k=0,\ldots, \ell$,  $\int_{\R^d}|x|^k|K(x)|dx<+\infty$ and 
\begin{equation} \label{eq: cancellation}
\int_{\R^d}x_{[1]}^k\,K(x)dx=\ldots = \int_{\R^d}x_{[d]}^k\,K(x)dx = {\bf 1}_{\{k = 0\}}\;\;\text{with}\;\;x=(x_{[1]},\ldots, x_{[d]}) \in \R^d.
\end{equation}
\end{defi}

\begin{remark} \label{rk: construction kernels}
A standard construction of $\ell$-order kernels is discussed in \cite{scott2015multivariate}, see also the classical textbook by Tsybakov \cite{tsybakov2009introduction}. In the numerical examples of the paper, we implement in dimension $d=1$ the Gaussian high-order kernels described in Wand and Schucany \cite{wand1990gaussian}. Given a $\ell$-regular kernel $\mathfrak h$ in dimension $1$, a multivariate extension is readily obtained by tensorisation: for $x=(x_{[1]},\ldots, x_{[d]}) \in \R^d$, set $K(x) = \mathfrak h^{\otimes d}(x) = \prod_{l = 1}^d \mathfrak h(x_{[l]})$, so that \eqref{eq: cancellation} is satisfied. 
\end{remark}

Finally, we pick a $\ell$-regular kernel $K$, a bandwidth $\eta >0$ and set
$$
\widehat \mu_T^{N,h,\eta}(x) := N^{-1} \sum_{n=1}^N \eta^{-d}K\big(\eta^{-1}(x-\bar X_T^{n,h})\big)
$$
for our estimator of $\mu_T(x)$. It is thus specified by the kernel $K$, the Euler scheme time step $h$, the number of particles $N$ and the statistical smoothing parameter $\eta$.

\subsection{Convergence results}

Under Assumptions \ref{ass: AI}, \ref{ass: AII} and (\ref{ass: AIII 1} (or later \ref{ass: AIII 2})), 
\eqref{original} 
admits a unique classical solution $(t,x) \mapsto \mu_t(x)$, see {\it e.g.} the classical textbook \cite{bogachev2022fokker}, Corollary 8.2.2. 
Our first result is a deviation inequality for the error $\widehat \mu_{T}^{N,h,\eta}(x)-\mu_{T}(x)$. We need some further notation.

\begin{defi}\label{def:abstract-euler-scheme}
The abstract Euler scheme relative to the McKean-Vlasov equation \eqref{mckeaneq} is defined as
\begin{align}\label{eq: euler abstract}
\begin{cases}
\bar{X}_{t_{m+1}}^h=\bar{X}_{t_{m}}^h+h \cdot b(t_m, \bar{X}_{t_{m}}^h, \mu_{t_{m}})+\sqrt{h} \cdot \sigma (t_m, \bar{X}_{t_{m}}^h) Z_{m+1},\\
Z_{m+1}=\frac{1}{\sqrt{h}}(B_{t_{m+1}}-B_{t_{m}}) \sim \mathcal{N}(0,1),\\
\mathcal L(\bar{X}_{t_{0}}^h)=\mu_0, 
\end{cases}
\end{align}
appended with its continuous version: for every $0\leq m\leq M-1$, $t\in[t_m, t_{m+1})$ and $1\leq n\leq N$:
\begin{align*}
\label{eq: abstract conteuler}
 \bar{X}_{t}^{h}  =\bar{X}_{t_{m}}^{h} + (t-t_m) b(t_m, \bar{X}_{t_{m}}^h, \mu_{t_{m}}) + \sigma (t_m, \bar{X}_{t_{m}}^h) (B_t-B_{t_m}).
\end{align*}
\end{defi}
We set $\bar \mu_t^h = \mathcal L(\bar X_t^h)$.
The term {\it abstract} for this version of the Euler scheme comes from its explicit dependence upon $\mu_{t_m}$, that prohibits its use in practical simulation. It however appears as a natural approximating quantity. 

\begin{defi}
For $\varepsilon >0$, the $\varepsilon$-accuracy of the abstract Euler scheme is
\begin{equation} \label{eq: def h opt}
h^\star(\varepsilon) = \sup\big\{h>0, \;\sup_{x \in \mathbb{R}^d}|\bar \mu_T^h(x)-\mu_T(x)| \leq \varepsilon\big\}, 
\end{equation}
with the convention $\sup \emptyset = \infty$. 
\end{defi}

Proposition \ref{densityeuler} below,  relying on the estimates of \cite{Gobet2008sharp}, implies that $\sup_{x \in \mathbb{R}^d}|\bar \mu_T^h(x)-\mu_T(x)|$  is of order $h$. Therefore 
$h^\star(\varepsilon)$  is well defined and positive for every $\varepsilon >0$.

We further abbreviate $K_\eta(x) = \eta^{-d}K(\eta^{-1}x)$ for $\eta >0$.
\begin{defi}\label{def:bias}
The bias (relative to a $\ell$-regular kernel $K$) of a function $\varphi:\R^d \rightarrow \R$ at $x$ and scale $\eta >0$ is 
$$\mathcal B_\eta(\varphi,x)=\sup_{0 < \eta'\leq \eta}\Big|\int_{\R^d}K_{\eta'}(x-y)\varphi(y)dy-\varphi(x)\Big|$$
and for $\varepsilon >0$, the $\varepsilon$-accuracy of the bias of $\varphi$ is
\begin{equation} \label{eq: def tau opt}
\eta^\star(\varepsilon,  \varphi) = \sup\big\{\eta>0, \mathcal B_\eta(\varphi, x) \leq \varepsilon\big\}.
\end{equation}
\end{defi}
Note that whenever $\varphi$ is continuous in a vicinity of $x$, we have $\mathcal B_\eta(\varphi,x) \rightarrow 0$ as $\eta \rightarrow 0$, and $\eta^\star(\varepsilon, \varphi)$  is well-defined and positive for every $\varepsilon >0$.\\

\begin{theorem} \label{thm: deviation}
Work under Assumptions \ref{ass: AI}, \ref{ass: AII} and \ref{ass: AIII 1}. Let $x \in \R^d$, $N\geq 2$, $\varepsilon >0$ and $\mathfrak C>0$. 
For any $0 <  h < \min\big(\,h^\star(\tfrac{1}{3}|K|_{L^1}^{-1}\varepsilon\,), \,\mathfrak CN^{-1}\big)$ and $0 < \eta < \eta^\star(\tfrac{1}{3}\varepsilon, \mu_T)$, we have:
\begin{equation} \label{eq: main deviation}
\PP\Big(\big|\widehat \mu^{N,h,\eta}_{T}(x)-\mu_{T}(x)\big| \geq \varepsilon \Big) \leq 2\kappa_1 \exp\Big(-\frac{  \kappa_2 N\varepsilon^2}{|K_\eta(x-\cdot)|^2_{L^2(\mu_{T})}+|K_\eta|_\infty\varepsilon}\Big),
\end{equation}
where $\kappa_1$ is a constant depending on $\mathfrak{C}$  and $\Cmodel$, and $\kappa_2$ depends only on $\Cmodel$.
\end{theorem}

Several remarks are in order:  {\bf 1)} Let $x \in \mathcal K \subset \R^d$, where $\mathcal K$ is an arbitrary compact. Then, 
$$|K_\eta(x-\cdot)|_{L^2(\mu_{T})}^2 \leq \sup_{(x,y) \text{ s.t. } x\,\in\,\mathcal{K}, \, x-y \in \mathrm{Supp}(K)}\mu_{T}(y)|K_\eta(x-\cdot)|_{L^2}^2 = C(\mu_{T},  \mathcal K , K)\eta^{-d}$$  and $|K_\eta|_\infty = \eta^{-d}|K|_\infty$,
therefore \eqref{eq: main deviation} rather reads
$$\PP\Big(\big|\widehat \mu^{N,h,\eta}_{T}(x)-\mu_{T}(x)\big| \geq \varepsilon \Big) \leq c_1 \exp\Big(-c_2 N\eta^{d}\frac{\varepsilon^2}{1+\varepsilon}\Big),$$
up to constants $c_1$ and $c_2$ that only depend on $\mathfrak C, \, \mathcal K,\,T$, the kernel $K$ and the model constant $\mathfrak{M}$.
{\bf 2)} The factor $\tfrac{1}{3}$ in the upper bound of $h$ and  $\eta$  can be replaced by an arbitrary constant in $(0,\frac{1}{2})$ by modifying the union bound argument \eqref{eq: union bound argument} in the proof. This modification only affects the multiplicative constant in the prefactor of the convergence rate in $N$ of the forthcoming Theorem~\ref{thm: error 1}. {\bf 3)} The estimate \eqref{eq: main deviation} gives an exponential bound of the form $\exp(-c N\eta^d\varepsilon^2)$ for the behaviour of the error $\widehat \mu^{N,h,\eta}_{ T}(x)-\mu_{ T}(x)$ for small $\varepsilon$, provided $h$ and $\eta$ are  sufficiently small. This is quite satisfactory in terms of statistical accuracy, for instance if one wants to implement confidence bands: for any risk level $\alpha$, with probability bigger than $1-\alpha$, the error $|\widehat \mu^{N,h,\eta}_{T}(x)-\mu_{T}(x)|$ is controlled by a constant times $\sqrt{N^{-1}\eta^{d}\log\tfrac{1}{\alpha}}$. {\bf 4)} We need $Nh$ bounded. This stems from the key estimate \eqref{eq: key estimate} in the proof of Proposition \ref{expp} 
where we need to control $N$ times the squared Wasserstein distance between the solution and the distribution of its Euler scheme which is of order $h$, in order to apply Novikov's criterion for an appropriate change of measure, see the proof for the details.  {\bf 5)} For the precise definitions of the constants $\kappa_1$ and $\kappa_2$ in \eqref{eq: main deviation}, we refer the reader to Proposition~\ref{prop: bernstein}. 
Note that although $\kappa_2$ depends only on $\Cmodel$, we retain the notation $\kappa_2$ for the subsequent discussion (see, e.g., the remarks following Theorem~\ref{thm: oracle}), which concerns the relation between the constants. \\

Theorem \ref{thm: deviation} does not give any quantitative information about the accuracy in terms of $h,N$ and $\eta$. Our next result gives an explicit upper bound for the expected error of order 
$p \geq 1$.

\begin{theorem} \label{thm: error 1}
Work under Assumptions \ref{ass: AI}, \ref{ass: AII} and \ref{ass: AIII 1}. Let $\mathfrak C >0$ and assume $Nh \leq \mathfrak C$. For every $p \geq 1$ and  $x \in \mathcal K \subset \R^d$, where $\mathcal K$ is an arbitrary compact, we have
\begin{equation} \label{eq: error lp}
\mathbb E\Big[\big|\widehat \mu^{N,h,\eta}_{T}(x)-\mu_{T}(x)\big|^p\Big] \leq \kappa_3 \big(\mathcal B_\eta(\mu_{T}, x)^p+(N\eta^d)^{-p/2}+h^p\big).
\end{equation}
where $\mathcal B_\eta$ is defined in Definition \ref{def:bias}, and $\kappa_3 >0$ depends  on $\mathfrak C$, $p$,  $\mathcal K$,  $|K|_{L^2}$, $|K|_{\infty}$ and  the model constant $\Cmodel$.  
\end{theorem}
 
If $z \mapsto \mu_{T}(z)$ is $k$-times continuously differentiable in a vicinity of $x$, then, for a regular kernel of order $(k-1)_+$, we have $\mathcal B_{\eta}(\mu_{T}, x) \lesssim \eta^k$, see for instance the proof of Corollary \ref{cor: error}. This enables one to optimise the choice of the bandwidth $\eta$:
 
\begin{coro} \label{cor: error}
Assume that $x \mapsto \mu_{T}(x)$ is $\mathcal C^k_b$ for some $k  \geq 1$ and that $K$ is $\ell$-regular, with $\ell \geq 0$. In the setting of Theorem \ref{thm: error 1}, the optimisation of the bandwidth choice $\eta = N^{-1/(2\min(k, \ell+1)+d)}$ yields the explicit error rate
\begin{equation} \label{eq: error lp opti}
\mathbb E\Big[\big|\widehat \mu^{N,h,\eta}_{T}(x)-\mu_{T}(x)\big|^p\Big] \leq \kappa_4 \big(N^{-\min(k, \ell+1)p/(2\min(k, \ell+1)+d)}+h^p\big),
\end{equation}
for some $\kappa_4>0$ depending on $\mathfrak C$, $p, \ell$,  $\mathcal K$, $ |K|_{L^2}$, $|K|_{\infty}$ and  the model constant $\Cmodel$. 
\end{coro}

Some remarks: {\bf 1)} The dependence of $\kappa_3$ and $\kappa_4$ on $(\mu_0,b,\sigma)$  can be computed explicitly  via the bounds that appear in Assumptions \ref{ass: AI}, \ref{ass: AII} and \ref{ass: AIII 1}. {\bf 2)} Corollary \ref{cor: error} improves the result of \cite[Theorem 3.1]{MR1910635} for the strong error of the classical solution $\mu_{T}(x)$ in the following way: in \cite{MR1910635}, the authors restrict themselves to the one-dimensional case $d=1$ with $p=1$ in the loss and take a Gaussian kernel for the density estimation, with bandwidth $h=\frac{T}{M}$ as in the time discretisation. They obtain the rate (in dimension $d=1$)
$$N^{-1/2}h^{-1/2-\epsilon}+h,$$
for arbitrary $\epsilon >0$ (up to an inflation in the constant when $\epsilon$ vanishes), and this has to be compared in our case to the rate \eqref{eq: error lp opti} which gives
$$  N^{-\frac{k}{2k+1}}+h = N^{-\frac{1}{2}}N^{\frac{1}{2(2k+1)}}+h, $$ 
which yields possible improvement depending on the regime for $h$, namely whenever $N \ll h^{-(2k+1)}$ which is less demanding as $k$ increases.  The essence of the improvement compared to \cite{MR1910635} is that we make an optimal bias-variance analysis based on nonparametric statistics techniques. In particular, the introduction of kernels that oscillate more than the Gaussian kernel enables us to obtain an error bound linked to the  smoothness parameter $k$ that can be very large whenever the classical solution $x \mapsto \mu_t(x)$ is smooth.\\

One defect of the upper bound \eqref{eq: error lp opti} is that the optimal choice of $\eta$ depends on the analysis of the bias $\mathcal B_{\eta}(\mu_{T}, x)^p$, and more specifically on the smoothness $k$, and that quantity is usually unknown or difficult to compute. Indeed, the information that $x \mapsto \mu_{T}(x)$ is $k$-times continuously differentiable with the best possible $k$ is hardly tractable from the data  $(\mu_0, b, \sigma)$, essentially due to the nonlinearity of the Fokker-Planck equation \eqref{original}. Finding an optimal $\eta$ without prior knowledge on the bias is a long standing issue in nonparametric statistics. One way to circumvent this issue is to select $\eta$ depending on the stochastic particle system $(\bar X^{1,h}_{T},\ldots, \bar X^{N,h}_{T})$ itself, in order to optimise the trade-off between the bias and the variance term. While this is a customary approach in statistics, to the best of our knowledge, this is not the case in numerical simulations. It becomes particularly suitable in the case of nonlinear McKean-Vlasov type models.\\

We adapt a variant of the classical Goldenshluger-Lepski's method \cite{GL08, GL11, GL14}, and refer the reader to classical textbooks such as \cite{gine2021mathematical}.  See also the  illuminating  paper by \cite{lacour2017estimator} that  gives  a good insight about ideas around bandwidth comparison and data driven selection. For simplicity, we focus on controlling the error for $p=2$. Fix $x\in \R^d$. Pick a finite set $\mathcal H \subset [N^{-1/d}(\log N)^{2/d}, 1]$, and define
$$\mathsf{A}_\eta^N(T,x)= \max_{\eta' \leq \eta, \eta' \in \mathcal H}\Big\{\big( \,\widehat \mu^{N,h,\eta}_{T}(x)-\widehat \mu^{N,h,\eta'}_{T}(x)\big)^2-(\mathsf V_{\eta}^N+\mathsf V_{\eta'}^N)\Big\}_+,$$ 
where $\{x\}_+ = \max(x,0)$ and 
\begin{equation} \label{eq: def variance lepski}
\mathsf V_{\eta}^N = \varpi |K|_{L^2}^2(\log N)N^{-1}\eta^{-d},\;\;\varpi >0.
\end{equation}
Let
\begin{equation} \label{eq: def minimisation}
\widehat \eta^N(T,x) \in \mathrm{argmin}_{\eta \in \mathcal H}(\mathsf{A}^N_\eta(T,x)+\mathsf{V}_\eta^N).
\end{equation}
We then have a so-called {\it oracle inequality}, following the standard terminology in nonparametric statistics. It means that our error bound is comparable to the optimal trade off between bias and variance in our method, namely $\min_{\eta \in \mathcal H}\big(\mathcal B_\eta(\mu_{T}, x)^2+N^{-1}\eta^{-d}\big)$ which is unknown to us, since it explicitly depends on $\mu_{T}$ in a vicinity of $x$. We also have an additional term of order $h^2$ linked to the Euler scheme. More precisely, we have
\begin{theorem} \label{thm: oracle}
Work under Assumptions \ref{ass: AI}, \ref{ass: AII} and \ref{ass: AIII 1}. For every  $x \in \mathcal K \subset \R^d$, where $\mathcal K$ is an arbitrary compact, we have
\begin{equation} \label{eq: error lp}
\mathbb E\Big[\big(\widehat \mu^{N,h, \widehat \eta^N( T,x) }_{T}(x)-\mu_{T}(x)\big)^2\Big] \leq \kappa_5 \Big(\min_{\eta \in \mathcal H}\big(\mathcal B_\eta(\mu_{T}, x)^2+N^{-1}\eta^{-d}\log N\big)+h^2\Big),
\end{equation}
for some constant $\kappa_5 >0$ depending on  $\mathcal K$, $|K|_{\infty}$, $\mathrm{Supp}(K)$ and  the model constant $\Cmodel$, as soon as $\mathrm{Card}(\mathcal H) \leq N$,  $\varpi$ is large enough,   and $\eta^d \geq N^{-1}\log N \tfrac{3 \log 2 |K|_\infty^2}{2 |K|_{L^2}^2}$.
\end{theorem}

Some remarks: {\bf 1)} Up to an unavoidable $\log N$-factor, known as the Lepski-Low phenomenon (see \cite{LEPSKI, LOW}) in statistics, we thus have a data-driven smoothing parameter $\widehat \eta^N(T,x)$ that automatically achieves the optimal bias-variance balance, without affecting the effect of the Euler discretisation step $h$. {\bf 2)}  The detailed constraint on $\varpi$ is provided in the proof of Theorem~2.9.  A limitation of the method is the choice of the pre-factor $\varpi$ in the bandwidth selection procedure, which depends on upper bound of $\mu_{T}$ locally around $x$, and more worryingly, on the constant $\kappa_2^{-1}$ of Proposition \ref{prop: bernstein} below which is quite large. In practice, and this is universal to all smoothing statistical methods, we have to adjust a specific numerical protocol, see Section \ref{sec: simulations} below.\\

When we translate this result in terms of number of derivatives for $x \mapsto \mu_{T}(x)$, we obtain the following adaptive estimation result.

\begin{coro} \label{cor: oracle adaptive}
Assume that $x \mapsto \mu_{T}(x)$ is $\mathcal C^k_b$ for some $k \geq 1$ and that $K$ is $\ell$-regular. Specify the oracle estimator $\widehat \mu_{T}^{N,h, \widehat \eta^N(T,x) }(x)$ with 
$$\mathcal H = \{ (N/\log N)^{-1/(2m+d)}, m=1,\ldots, \ell+1\}.$$
In the setting of Theorem \ref{thm: oracle}, for every  $x \in \mathcal K \subset \R^d$, where $\mathcal K$ is an arbitrary compact, we have
$$
\mathbb E\Big[\big(\widehat \mu_T^{N,h, \widehat \eta^N(T,x) }(x)-\mu_{T}(x)\big)^2\Big] \leq \kappa_6 \Big(\Big(\frac{\log N}{N}\Big)^{\frac{2\min(k,\ell+1)}{2\min(k,\ell+1)+d}}+h^2\Big),
$$
for some $\kappa_6>0$ depending on $\mathcal{K}, \,|K|_{\infty}$, $\mathrm{Supp}(K)$,   $k$ and the model constant $\Cmodel$. 
\end{coro}

In practice, see in particular Section \ref{sec: simulations} below, if $x \mapsto \mu_{T}(x)$ is very smooth (and this is the case in particular if $x \mapsto b(t,x,\mu_t)$ is smooth and $\sigma$ constant) we are limited in the rate of convergence by the order $\ell$ of the kernel. This shows in particular that it is probably not advisable in such cases to pick a Gaussian kernel for which we have the restriction $\ell=1$.

\subsection{Nonlinearity of the drift in the measure argument} \label{sec: nonlinear}

In the  case  of a drift $(t,x,\mu) \mapsto b(t,x,\mu)$ with a nonlinear dependence in the measure argument $\mu$, the assumptions are a bit more involved. For a smooth real-valued test function $\varphi$ defined on $\R^d$, we set
 \begin{equation} \label{eq: generator}
 \mathcal A_t \varphi(x) = b(t,x,\mu_t)^\top\nabla \varphi(x)+\tfrac{1}{2}\sum_{k,k'=1}^d c_{kk'}(t,x)\partial_{kk'}^2 \varphi(x),
 \end{equation}
where $c_{kk'}(t,x)  = (\sigma \sigma^\top)_{kk'}(t,x)$, and that can also be interpreted  as the generator of the associated nonlinear Markov process of $(X_t)_{t \in [0,T]}$ defined in \eqref{mckeaneq}.  Following \cite{MR3752669}, we say that a mapping $f:\mathcal P_1 \rightarrow \R^d$ has a linear functional derivative if there exists $\delta_\mu f:\R^d\times \mathcal P_1 \rightarrow 
 \R^d$ such that
 $$f(\mu')-f(\mu) = \int_0^1 d\vartheta \int_{\R^d} \delta_\mu f(y, (1-\vartheta)\mu+\vartheta\mu')(\mu'-\mu)(dy),$$
with
 $|\delta_\mu f(y',\mu')-\delta_\mu f(y,\mu)| \leq C(|y'-y|+\mathcal W_1(\mu',\mu))$ and $|\partial_y(\delta_\mu f(y,\mu')-\delta_\mu f(y,\mu))| \leq C\mathcal W_1(\mu',\mu)$, for some $C \geq 0$. We can iterate the process via mappings $\delta_\mu^\ell f: (\R^d)^\ell \times \mathcal P_1 \rightarrow \R^d$ for $\ell = 1,\ldots, k$ defined recursively by $\delta_\mu^\ell f = \delta_\mu \circ \delta_\mu^{\ell-1} f$.
 
 \begin{assump}(Nonlinear representation of the drift.)
\label{ass: AIII 2}
For every $\mu \in \mathcal P_1$, $(t,x) \mapsto b(t,x,\mu)$ is $\mathcal C_b^{1,3}$. Moreover, for every $x$,
$(t,y) \mapsto \mathcal A_t \delta_\mu b(t,x,y,\mu_t)$ exists and is continuous and bounded.\\ 

Finally, $\mu \mapsto b(t,x,\mu)$ admits a $k$-linear functional derivative (with $k \geq d$ for $d=1,2$ and $k \geq d/2$ for $d \geq 3$) that admits the following representation
\begin{equation} \label{eq: smoothness nonlinear b}
\delta_\mu^k b(t,x,(y_1,\ldots, y_k),\mu) = \sum_{\mathcal I, m} \bigotimes_{j \in \mathcal I}(\delta_\mu^k b)_{\mathcal I,j, m}(t,x,y_j,\mu),
\end{equation}
where the sum in $\mathcal I$ ranges over subsets of $\{1,\ldots, k\}$, the sum in $m$ is finite and the mappings $(x,y) \mapsto (\delta_\mu^k b)_{\mathcal I,j, m}(t,x,y_j,\mu)$ are Lipschitz continuous, uniformly in $t$ and $\mu$.
\end{assump}
Assumption \ref{ass: AIII 2} has two objectives: first comply with the property that $(t,x) \mapsto b(t,x,\mu_t)$ is $\mathcal C^{1,3}_b$ in order to apply the result of Gobet and Labart \cite{Gobet2008sharp} in Proposition \ref{densityeuler} below and second, provide a sufficiently smooth structure in \eqref{eq: smoothness nonlinear b} in order to implement the change of probability argument of Proposition \ref{prop: bernstein}. While \eqref{eq: smoothness nonlinear b} appears a bit {\it ad-hoc} and technical, it is sufficiently general to encompass drift of the form
$$b(t,x,\mu) = \sum_j F_j\big(t,x,\int_{\R^{q_2}} G_j\big(t,x,\int_{\R^d} H_j(t,x,z)\mu(dz),z'\big)\lambda_j(dz')\big)$$ for smooth mappings $F_j(t,x,\cdot): \R^{q_3}\rightarrow \R^d,G_j(t,x,\cdot) : \R^{q_1} \times \R^{q_2} \rightarrow \R^{q_3}, H_j(t,x,\cdot):\R^d \rightarrow \R^{q_1}$ and positive measures $\lambda_j$ on $\R^{q_2}$ in some cases and combinations of these, see Jourdain and Tse \cite{jourdain2020central}. Explicit examples where the structure of the drift is of the form \eqref{eq: smoothness nonlinear b} rather than the form in Assumption \ref{ass: AIII 1} are given for instance in \cite{coghi2018pathwise,  MR779460, meleard1996asymptotic, MJ98}.

\begin{theorem} \label{thm: nonlinear}
Work under Assumptions \ref{ass: AI}, \ref{ass: AII} and \ref{ass: AIII 2}. The results of Theorems \ref{thm: deviation}, \ref{thm: error 1}, \ref{thm: oracle} and Corollary \ref{cor: error}, \ref{cor: oracle adaptive} remain valid, up to a modification of the constants $\kappa_i$, $i=1,\ldots, 6$.
 \end{theorem}

\section{Numerical implementation} \label{sec: simulations}

We investigate our simulation method on three different examples. The simulation code is available via 
 \texttt{GitHub} (see: \href{https://bit.ly/3UrOmiw}{\texttt{bit.ly/3UrOmiw}}). 
\begin{enumerate}
\item 
{\it A linear interaction as in Assumption \ref{ass: AIII 1} of the form $\widetilde b(t,x,y) = c(x-y)$}. This case is central to several applications (see the references in the introduction for instance) and has moreover the advantage to yield an explicit solution for $\mu_t(x)$, enabling us to accurately estimate the error of the simulation. 

\item 
{\it A double layer potential with a possible singular shock in the common drift that ``stresses" Assumption \ref{ass: AIII 1}}. Although the solution is not explicit, the singularity enables us to investigate the automated bandwidth choice on repeated samples and sheds some light on the effect of our statistical adaptive method.  

\item  
{\it The Burgers equation in dimension $d=1$}. Although not formally within the reach Assumption \ref{ass: AIII 1}, we may still implement our method. While we cannot  provide  theoretical guarantees, we again have an explicit solution for $\mu_t(x)$ and we can accurately measure the performance of our simulation method.   

\end{enumerate}
For simplicity, all the numerical experiments are conducted in dimension $d=1$. The common findings of our numerical investigations can be summarised as follows: when the density solution $x \mapsto \mu_t(x)$ is smooth, high-order kernels (with vanishing moments beyond order $1$) outperform classical kernels such as the Epanechnikov kernel or a standard Gaussian kernel as shown in Cases 1 and 3. Specifically, we implement high-order Gaussian based kernels as developed for instance in Wand and Schucany \cite{wand1990gaussian}. Table 1 below gives explicit formulae for $K$ and $\ell = 1, 3, 5, 7, 9$ (Recall Remark \ref{rk: construction kernels} for the construction of $\ell$-order kernels  for $d>1$). Otherwise, for instance when the transport term has a common Lipschitz continuous component (but  not more regular than Lipschitz continuous), the automated bandwidth method tends to adapt the window to prevent over-smoothing, even with high-order kernels. Overall, we find in all three examples that implementing high-order kernels is beneficial for the quality of the simulation. 

\vspace{0.5cm}
\begin{table}[ht]
\begin{center}
\begin{tabular}{| c | l |} 
\hline
 order $\ell$ & kernel  $K(x)$  \\ 
 \hline
 1 &   $\phi(x)$ \\ 
$3$ &  $\tfrac{1}{2}(3-x^2)\phi(x)$  \\ 
 $5$ &  $\tfrac{1}{8}(15-10x^2+x^4)\phi(x)$  \\ 
 $7$ &  $\tfrac{1}{48}(105-105x^2+21x^4-x^6)\phi(x)$  \\ 
 $9$ & $\tfrac{1}{384}(945-1260x^2+378x^4-36x^6+x^8)\phi(x)$    \\  [2pt]
 \hline
\end{tabular}
\vspace{0.5cm}
\caption{\label{demo-table}Construction of kernels $K$ of order $\ell = 1, 3, 5, 7, 9$, with $\phi(x) = (2\pi)^{-1/2}\exp(-\tfrac{1}{2}x^2)$, dimension $d=1$, following \cite{wand1990gaussian}.}
\end{center}
\end{table}


\subsection{The case of a linear interaction}
We consider the simplest McKean-Vlasov SDE with linear interaction
$$dX_t = -\int_{\R}(X_t - x)\mu_t(dx)dt+dB_t,\;\;\mathcal L(X_0) = \mathcal N(3, \tfrac{1}{2})$$
where $(B_t)_{t \in [0,T]}$ is a standard Brownian motion and $\mathcal N(\rho,\sigma^2)$ denotes the Gaussian distribution with mean $\rho$ and variance $\sigma^2$. We do not fulfill Assumption \ref{ass: AIII 1} here since $\widetilde{b}(t,x,y)=x-y$ is not bounded. We nevertheless choose this model for simulation since it has an explicit solution $\mu_t$ as a stationary Ornstein-Uhlenbeck process, namely, abusing notation slightly, $\mu_t = \mathcal L(X_t) = \mathcal N(3,\tfrac{1}{2})$.\\

We implement the Euler scheme $(\bar X_t^{1,h},\ldots, \bar X_t^{N,h})_{t \in [0,T]}$ defined in \eqref{discresys}  with $b(t,x,\mu) = -\int_{\R}(x-y)\mu(dy)$ and $\sigma(t,x) = 1$. We pick $T=1$, $h = 10^{-2}T = 10^{-2}$, for several values of the system size $N = 2^7 = 128 ,2^8 = 256,\ldots, 2^{15} = 32768$. We then compute
$$\widehat \mu_T^{N,h,\eta}(x) := N^{-1}\sum_{n = 1}^N \eta^{-d}K\big(\eta^{-1}(x-\bar X_T^{n,h})\big)$$
as our proxy of $\mu_T(x)$ for $\ell = 1, 3, 5, 7, 9$ and the kernels given in Table \ref{demo-table}. We pick $\eta = N^{-1/(2(\ell+1)+1)}$ according to Corollary \ref{cor: error} since $\mu_T \in \mathcal C^k$ for every $k \geq 1$. We repeat the experiment $30$ times to obtain independent Monte-Carlo proxies $(\widehat \mu_T^{N,h,\eta})_j(x)$ for $j=1,\ldots, 30$ and we finally compute the Monte-Carlo strong error
$\mathcal E_N = \frac{1}{30} \sum_{j = 1}^{30}\max_{x \in \mathcal D}\big|(\widehat \mu_T^{N,h,\eta})_j(x)-\mu_T(x)\big|^2,$ 
where $\mathcal D$ is a uniform grid of $1000$ points in $[0,6]$. 
(The domain is dictated by our choice of initial condition $\mu_0 \sim \mathcal N(3,\tfrac{1}{2})$; the probability that a random variable with this distribution falls outside $[0,6]$ is about $2.2 \times 10^{-5}$, which justifies this truncation.)\\

In Figure \ref{fig: error OU}, we display on a log-2 scale $\mathcal E_N$ for $\ell = 1, 3, 5, 7, 9$. In Figure \ref{fig: slopes OU} we display the least-square estimates of the slope of each curve of Figure \ref{fig: error OU} according to a linear model. We thus have a proxy of the rate $\alpha_\ell$ of the error $\mathcal E_N \approx N^{-\alpha_\ell}$ for different values of $\ell$. We clearly see that higher-order kernels are better suited for estimating $\mu_t$. This is of course no surprise from a statistical point of view, but this may have been overlooked in numerical probability simulations.

\begin{figure}[H]
\begin{center}
\includegraphics[width=0.6\linewidth]{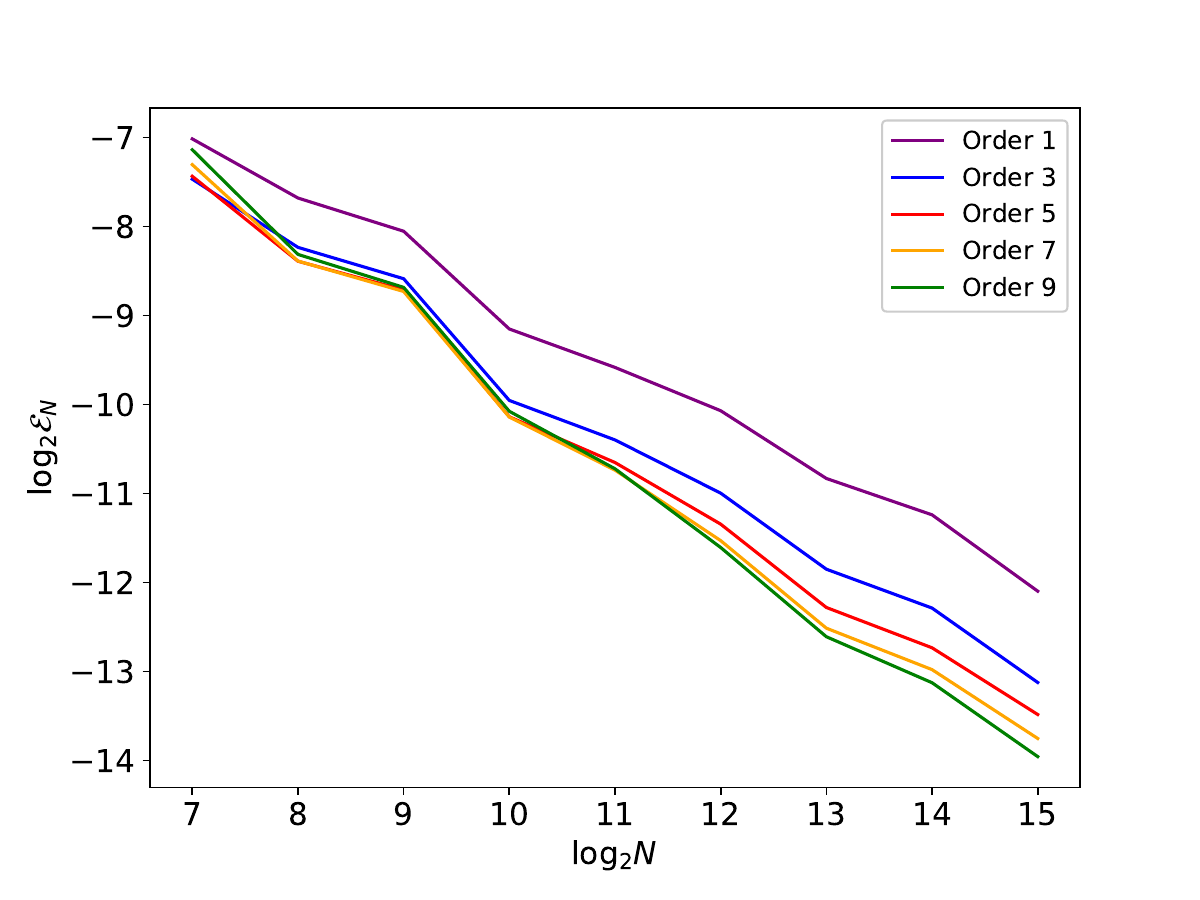} 
\vspace{-0.2cm}
\caption{\label{fig: error OU} {\it Monte-Carlo (for 30 repeated samples) strong error for different kernel orders: $\log_2 \mathcal E_N$ as a function of $\log_2 N$ for $\ell = 1$ (purple), $\ell = 3$ (blue), $\ell = 5$ (red), $\ell = 7$ (orange), $\ell = 9$ (green)}. We see that a polynomial error in $N$ is compatible with the data.} 
\end{center}
\end{figure}


\begin{figure}[H]
\begin{center}
\includegraphics[width=0.6\linewidth]{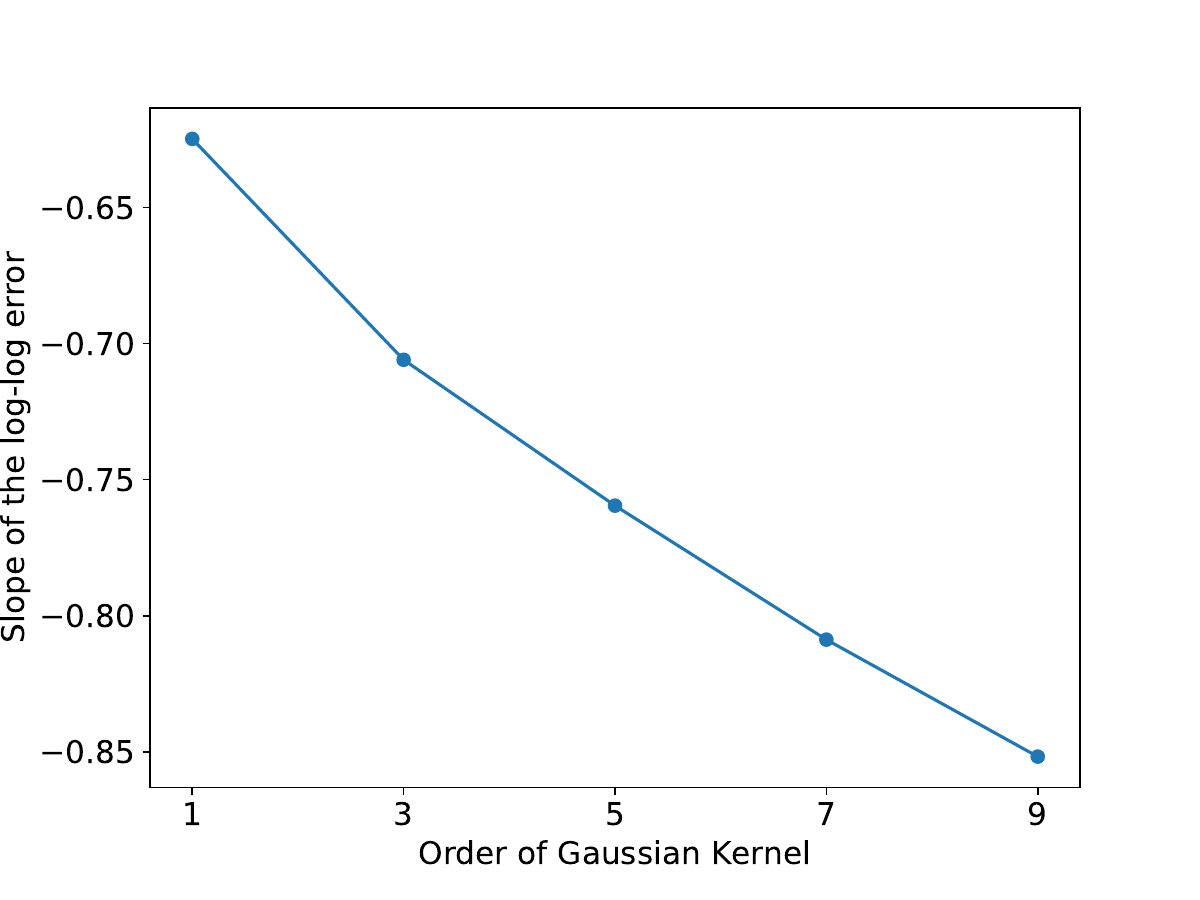} 
\vspace{-0.2cm}
\caption{\label{fig: slopes OU} {\it Least-square estimates of the slope $\alpha_\ell$ of $\log_2 \mathcal E_N = \alpha_\ell \log_2 N+\text{noise}$ in a linear model representation. We plot $\alpha_\ell$ as a function of the order $\ell$ of the kernel}. We see that a higher order $\ell$ for the choice of the kernel systematically improves on the error rate, as predicted by the statistical bias-variance analysis.}
\end{center}
\end{figure}


\subsection{A double layer potential}

We consider an interaction consisting of a smooth long range attractive and small range repulsive force, obtained as the derivative of double-layer Morse potential. Such models are commonly
used (in their kinetic version) in swarming modelling, see for instance \cite{bolley2011stochastic}. The corresponding McKean-Vlasov equation is 
$$dX_t = \int_{\R}U'(X_t - x)\mu_t(dx)dt+dB_t,\;\;\mathcal L(X_0) = \mathcal N(0,1),$$
where we pick $U(x) = -\exp(-x^2)+2\exp(-2x^2)$. The potential $U$ and its derivative $U'$ are displayed in Figure \ref{fig: potential plot}. 

We implement the Euler scheme $(\bar X_t^{1,h},\ldots, \bar X_t^{N,h})_{t \in [0,T]}$ defined in \eqref{discresys} 
for coefficients $b(x,\mu) = \int_{\R}U'(x-y)\mu(dy)$ and $\sigma(t,x) = 1$. We pick $T=1$, $h = 10^{-2}T = 10^{-2}$, for several values of the system size $N = 2^5 = 32 ,2^6 = 64,\ldots, 2^{16} = 65736$. We then compute
$$\widehat \mu_T^{N, h,  \widehat \eta^N(T,x)}(x) := N^{-1}\sum_{n = 1}^N  \widehat \eta^N(T,x)^{-d}K\big( \widehat \eta^N(T,x)^{-1}(x-\bar X_T^{n,h})\big)$$
as our proxy of $\mu_T(x)$ for $\ell = 1, 3, 5, 7, 9$ according to $K$ as in Table \ref{demo-table}. The data-driven bandwidth $\widehat \eta^N(T,x)$ is computed via the minimisation \eqref{eq: def minimisation}, for which one still needs to set the penalty parameter $\varpi$ arising in the Goldenschluger-Lepski method, see \eqref{eq: def variance lepski} in particular. The grid is set as $\mathcal H = \{ (N/\log N)^{-1/(2m+1)}, m=1,\ldots, \ell+1\},$ see in particular Corollary \ref{cor: oracle adaptive} in order to mimick the oracle.


\begin{figure}[H]
\begin{center}
\includegraphics[width=0.6\textwidth]{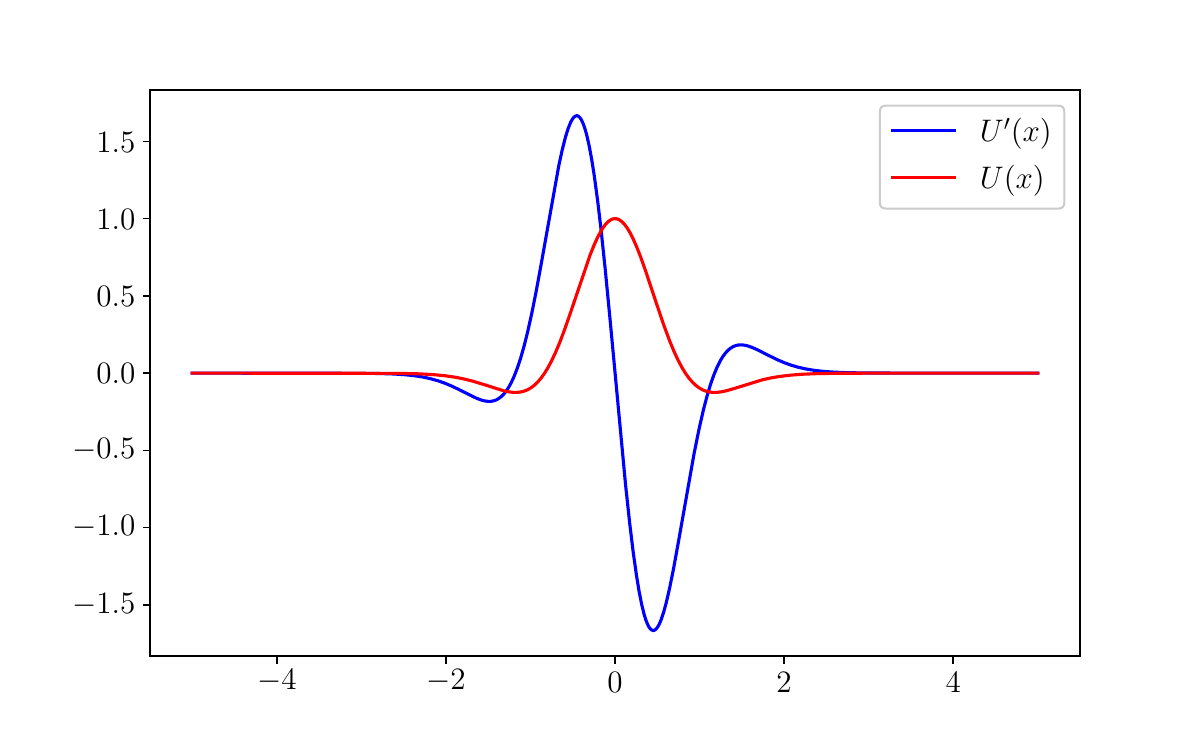} 
\vspace{-0.5cm}
\caption{\label{fig: potential plot} {\it Plot of the potential $U$ (red) and its derivative $U'$ (blue)}. }
\end{center}
\end{figure}

In this setting, we do not have access to the exact solution $\mu_T(x)$; we nevertheless explore several numerical aspects of our method via the following experiments:
\begin{itemize}
\item Investigate the effect of the order $\ell$ of the kernel (for an {\it ad-hoc} choice of the penalty parameter $\varpi$). We know beforehand that the mapping $x \mapsto \mu_T(x)$ is smooth, and we obtain numerical evidence that a higher-order kernel gives better results by comparing the obtained $\widehat \mu_T^{N, h,  \widehat \eta^N(T,x)}(x)$ for different values of $\ell$ as $N$ increases in the following sense: for $N=2^5$ particles, the estimates for high order kernels are closer to the estimates obtained for $N=2^{10}$ than lower order kernels.
\item Investigate the distribution of the data-driven bandwidth $ \widehat \eta^N(T,x)$, for repeated samples as $x$ varies and for different values of the penalty $\varpi$ over the grid $\mathcal H = \{ (N/\log N)^{-1/(2m+1)}, m=1,\ldots, \ell+1\}$. The estimator tends to pick the larger bandwidth with overwhelming probability, which is consistent with our prior knowledge that $x \mapsto \mu_T(x)$ is smooth.
\item In order to exclude an artifact from the preceding experiment, we conduct a cross-experiment by  perturbing the drift with an additional Lipschitz common force (but not smoother) 
that saturates our Assumption \ref{ass: AIII 1}. This extra transport term lowers down the smoothness of $x \mapsto \mu_T(x)$ and by repeating the preceding experiment, we obtain a different distribution of the data-driven bandwidth $ \widehat \eta^N(T,x)$, advocating in favour of a coherent oracle procedure.
\end{itemize}

\begin{figure}[H]
\begin{center}
\includegraphics[width=0.6\textwidth]{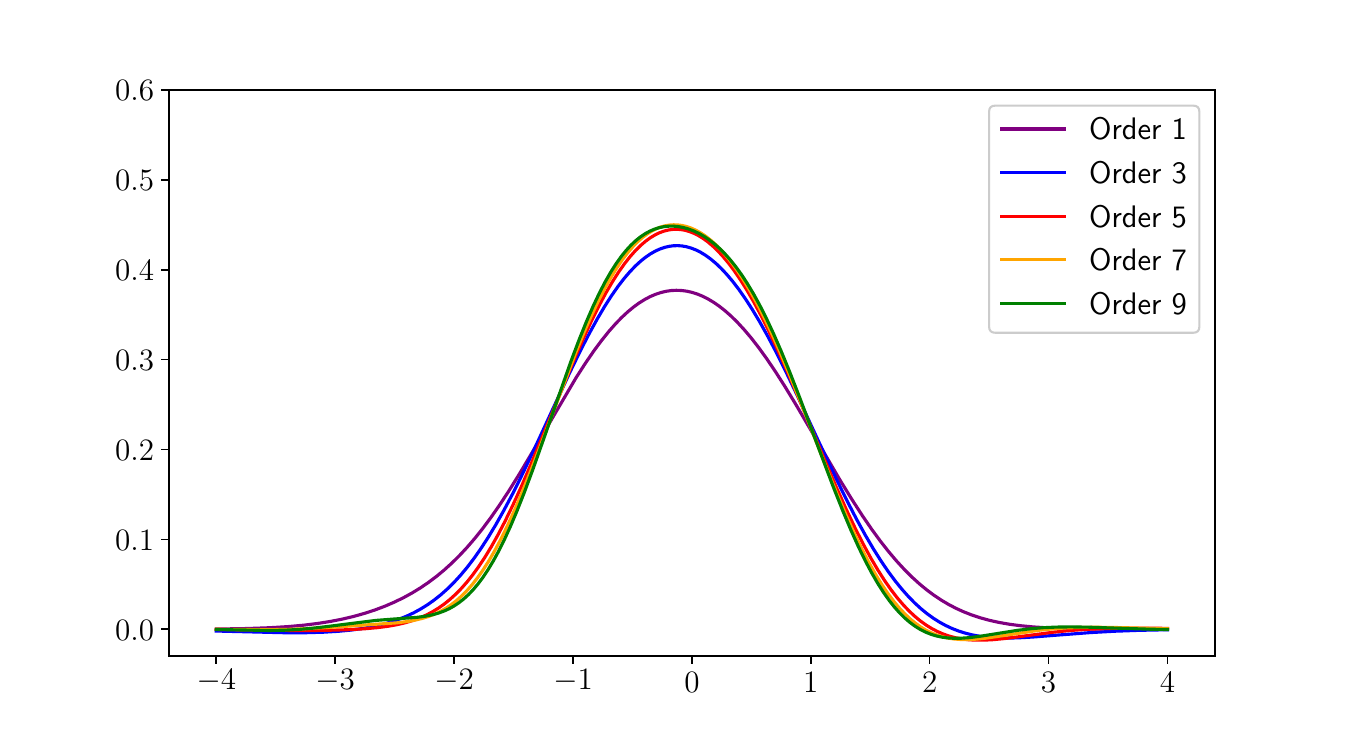}

\includegraphics[width=0.6\textwidth]{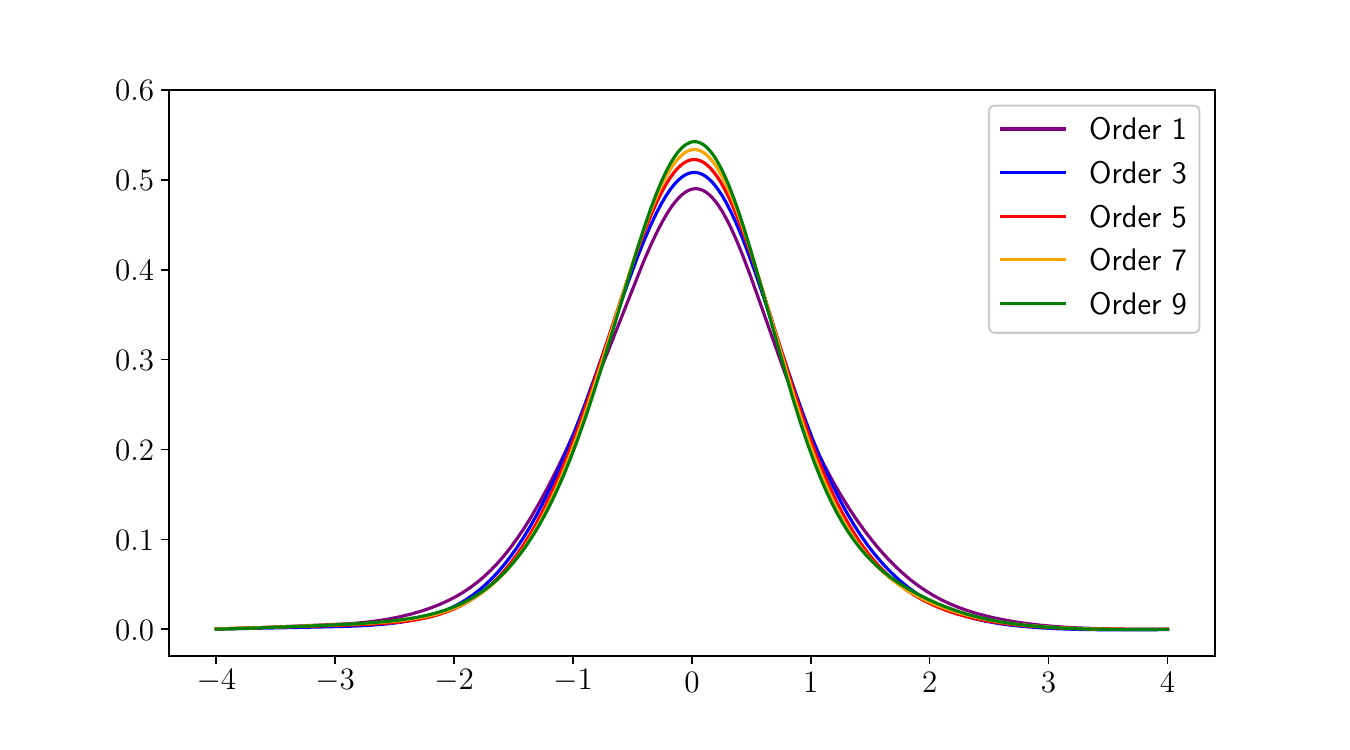}
\vspace{-0.5cm}
\caption{\label{fig: doublelayer_bandwidth_effect} {\it The graph of $x \mapsto \widehat \mu_t^{N, h,  \widehat \eta^N(t,x)}(x)$}. The domain $x \in [-4,4]$ is computed over a discrete grid of $2000$ points, {\it i.e.} mesh $4 \cdot  10^{-3}$) for $N = 2^5$ (Left) and $N = 2^{10}$ (Right). }
\end{center}
\end{figure}

\begin{figure}[H]
\begin{center}
\includegraphics[width=0.6\textwidth]{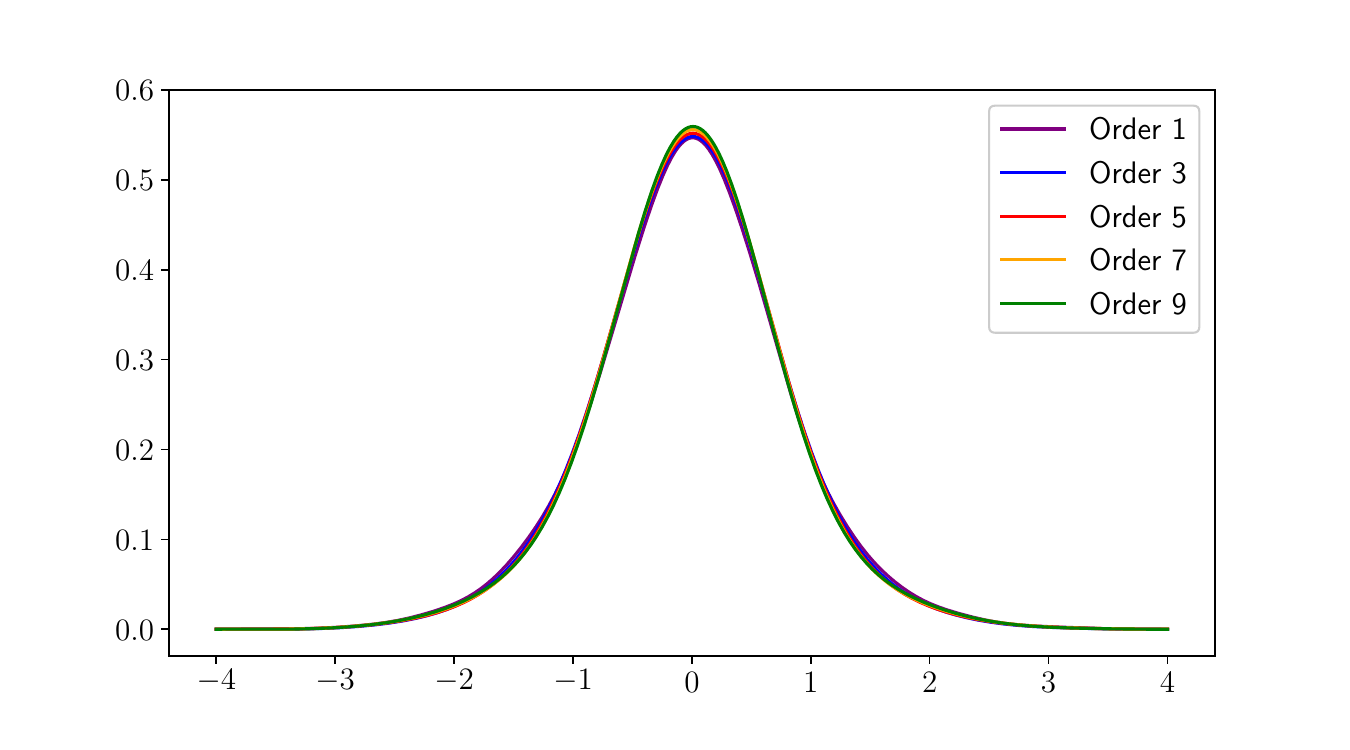} 
\vspace{-0.5cm}
\caption{\label{fig: doublelayer_bandwidth_ultimate} {\it Same experiment as in Figure \ref{fig: doublelayer_bandwidth_effect} mimicking an asymptotic behaviour of the procedure for $N=2^{15}$}.}
\end{center}
\end{figure}

\vspace{-0.5cm}

\subsubsection*{The effect of the order $\ell$ of the kernel}
We display in Figure \ref{fig: doublelayer_bandwidth_effect} the graph of $x \mapsto \widehat \mu_T^{N, h,  \widehat \eta^N(T,x)}(x)$ for $T=1$, $\ell = 1, 3, 5, 7, 9$ for $N=2^5$ and $N = 2^{10}$. The tuning parameter in the choice of $\widehat \eta^N(T,x)$ is set to
\footnote{This choice is quite arbitrary, the values of the data-driven bandwidths showing stability as soon as $\varpi \geq 20$.} 
$\varpi = 23$. The same experiment is displayed in Figure \ref{fig: doublelayer_bandwidth_ultimate}  for $N=2^{15}$.

The value $N=2^{15}$ mimicks the asymptotic performance of the procedure as compared to $N=2^5$ or $N=2^{10}$. We observe that the effect of the order of the kernel $\ell$ is less pronounced. A visual comparison of Figure \ref{fig: doublelayer_bandwidth_effect} and \ref{fig: doublelayer_bandwidth_ultimate} suggests that a computation with higher order kernels always perform better, since the shape obtained for $\ell = 9$ for small values of $N$ is closer to the asymptotic proxy $N = 2^{15}$ than the results obtained with smaller values of $\ell$.

\subsubsection*{The distribution of the data-driven bandwidth}

We pick several values for $\varpi$ (namely $\varpi  = 20^{-1}, 10^{-1}, 1, 10, 20$) and compute $\widehat \eta^N(T,x)$ accordingly for $1200$ samples.
Figure \ref{fig: distribution bandwidth} displays the histogram of the $\widehat \eta^N(T,x)$ for $x \in [-4, 4]$ (discrete grid with mesh $8\cdot 10^{-2}$) for $\ell = 3$ and $\ell = 5$. We observe that the distribution is peaked around large bandwidths at the far right of the spectrum of the histogram, with comparable results for $\ell = 7$ and $\ell = 9$. In Figure \ref{fig: restricted distribution bandwidth}, we repeat the experiment  for $\ell = 5$ over the restricted domain $[-\tfrac{1}{2}, \tfrac{1}{2}]$  (discrete grid, mesh $10^{-2}$) where we expect the solution $\mu_T(x)$ to be more concentrated and see no significant difference. This result is  in line with  the statistical nonparametric estimation of a smooth signal (see {\it e.g.} the classical textbook \cite{silverman2018density}). 

\subsubsection*{Cross experiment by reducing the smoothness of $x \mapsto \mu_{T}(x)$}

We repeat the previous experiment by adding a perturbation in the drift, via a common force $V'(x) = 2(1-|x|){\bf 1}_{\{|x| \leq 1\}}$. The drift now becomes $b(x,\mu) = V'(x)+\int_{\R}U'(x-y)\mu(dy)$. The common transport term $V'(x)$ is no smoother than Lipschitz continuous thus reducing the smoothness of the solution $x \mapsto \mu_T(x)$. In the same experimental conditions as before, Figure \ref{fig: restricted distribution bandwidth perturbed} displays the distribution of $\widehat \eta^N(T,x)$ and is to be compared with Figure \ref{fig: restricted distribution bandwidth}. We observe that the distribution is modified, in accordance with the behaviour of the oracle  bandwidth of a signal with lower order of smoothness. This advocates as further empirical evidence of the coherence of the method.

\vspace{1cm}

\begin{figure}[H]
\begin{center}
\includegraphics[width=0.55\textwidth]{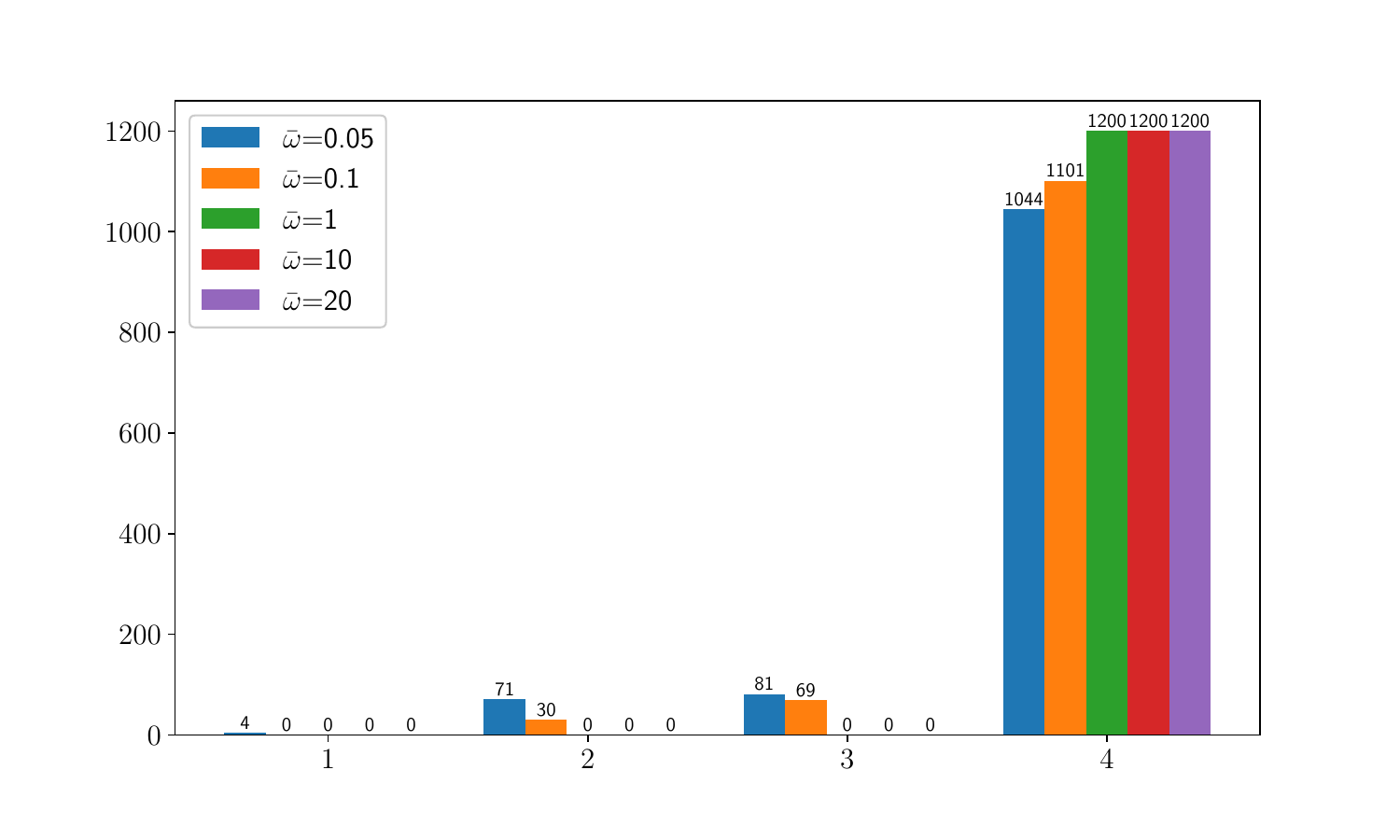}

\includegraphics[width=0.8\textwidth]{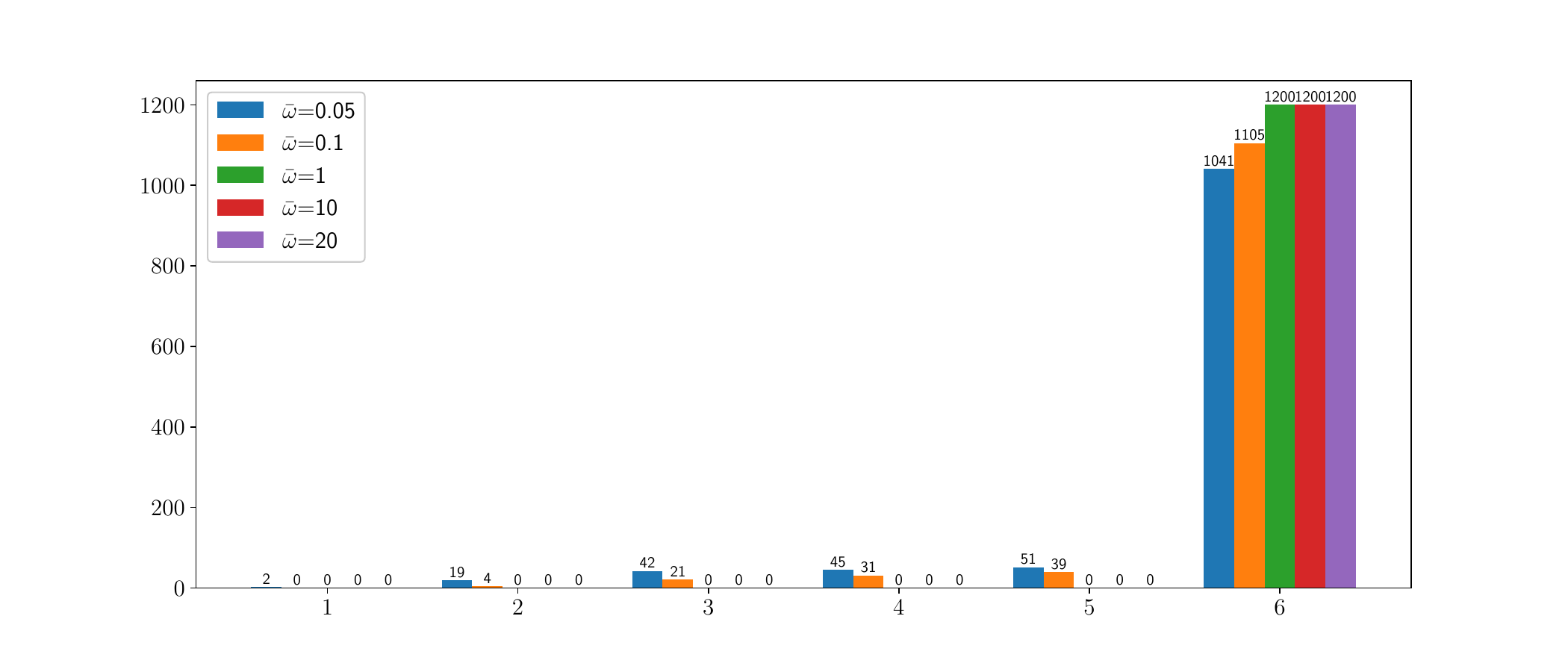}
\vspace{-0.5cm}
\caption{\label{fig: distribution bandwidth} {\it Distribution of $\widehat \eta^N(T,x)$}. The domain $x \in [-4,4]$ is computed over a discrete grid of $100$ points, {\it i.e.} mesh $8 \cdot 10^{-2}$), $N = 2^5, 2^6, \cdots , 2^{16}$ for $\ell =  3$ (\textit{Top}) and $\ell = 5$ (\textit{Bottom}).}
\end{center}
\end{figure}

\vspace{-1cm}

\begin{figure}[H]
\begin{center}
\includegraphics[width=0.8\textwidth]{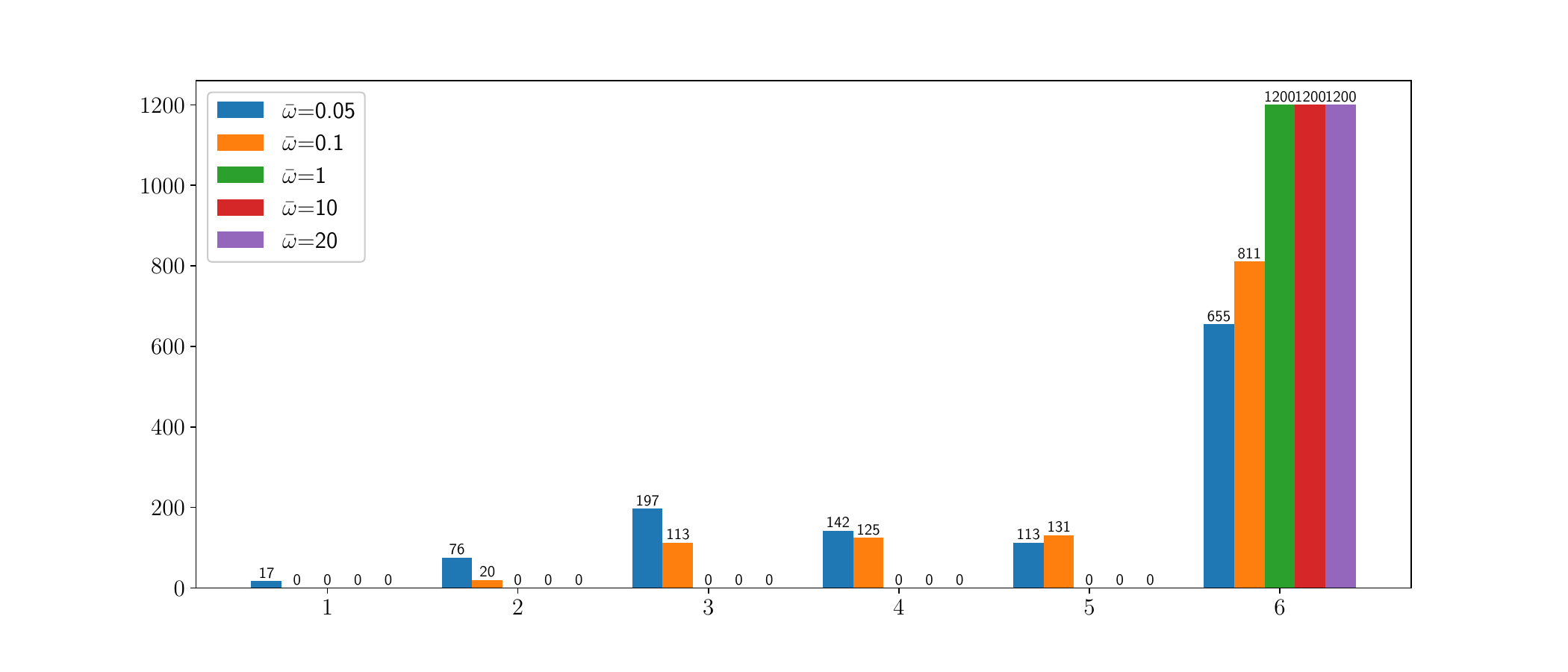} 
\vspace{-0.5cm}
\caption{\label{fig: restricted distribution bandwidth} {\it Same experiment as in Figure \ref{fig: distribution bandwidth} (\textit{Bottom}) on  the restricted domain $[-\tfrac{1}{2}, \tfrac{1}{2}]$ over a discrete grid of $100$ points}. The results are comparable with the experiment displayed in Figure \ref{fig: distribution bandwidth} (\textit{Bottom}).}
\end{center}
\end{figure}


\begin{figure}[H]
\begin{center}
\includegraphics[width=0.8\textwidth]{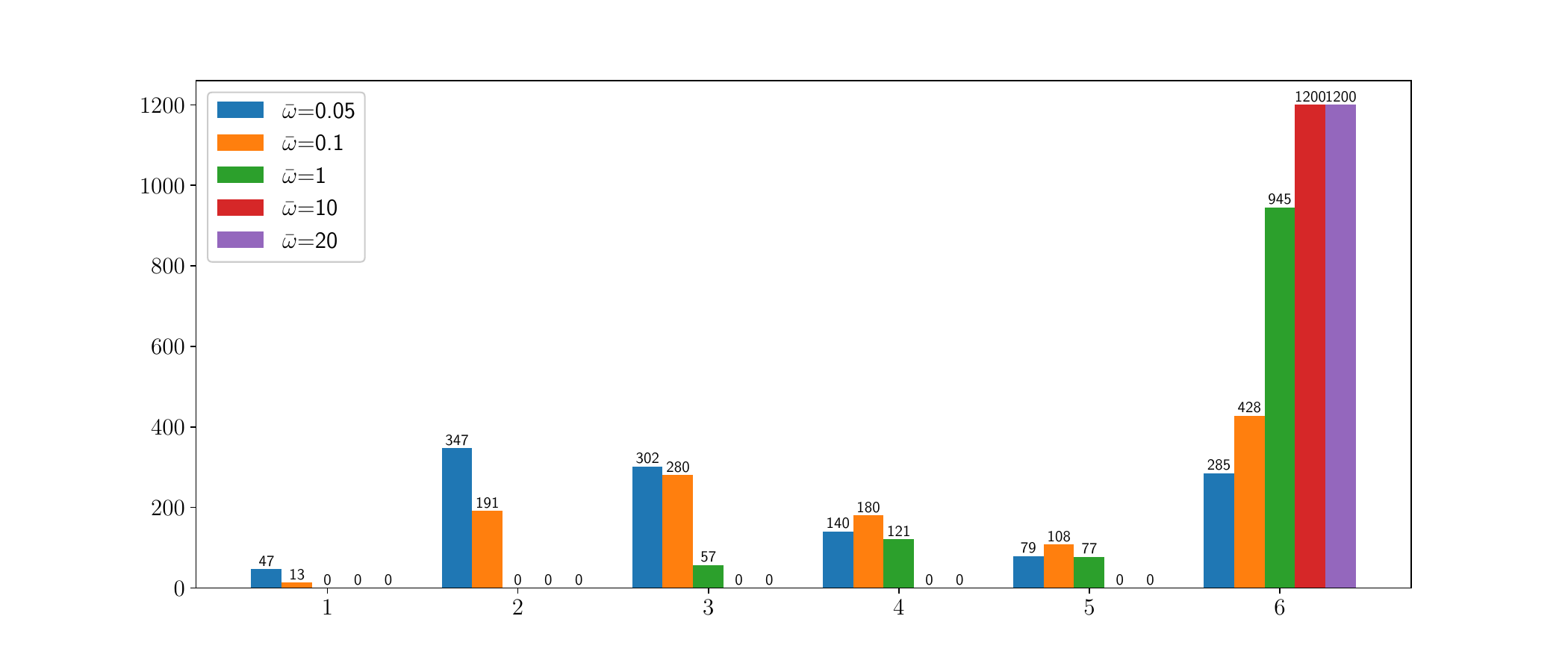} 
\vspace{-0.5cm}
\caption{\label{fig: restricted distribution bandwidth perturbed} {\it Same experiment as in Figure \ref{fig: restricted distribution bandwidth} on the restricted domain $[-\tfrac{1}{2}, \tfrac{1}{2}]$ over a discrete grid of $100$ points, when adding a common perturbation force $V'(x) = 2(1-|x|){\bf 1}_{\{|x| \leq 1\}}$ in the drift.} The effect on the empirical distribution of $\widehat \eta^N(T,x)$ is in line with the behaviour of an oracle bandwidth that adjusts smaller bandwidths when the signal is less smooth.}
\end{center}
\end{figure}
\vspace{-0.6cm}

\subsection{The Burgers equation in dimension $d=1$} \label{sec: burgers}
We consider here the following McKean-Vlasov equation 
$$dX_t = \int_{\R} H(x-y)\mu_t(dy)dt+\sigma dB_t,\;\;\mathcal L(X_0) = \delta_0,$$
with $\sigma =  \sqrt{0.2}$, associated with the Burgers equation in dimension $d=1$ for $H(x-y) = {\bf 1}_{\{y \leq x\}}$.  This model has already been introduced in \cite{bossy1997stochastic}, and the interacting particle system \eqref{discresys} matches  Equation~(2.3) of \cite{bossy1997stochastic}.  Although the discontinuity at $y=x$ rules out our Assumption \ref{ass: AIII 1}, we may nevertheless implement our method, since a closed form formula is available for  $\mu_t(x)$. More specifically, for $t>0$, the cumulative density function of $\mu_t(x)$ is explicitly given by 
$$M_t(x) = \frac{\int_0^\infty \exp\big(-\frac{1}{\sigma^2}\big(\frac{(x-y)^2}{2t}+y\big)\big)dy}{\int_{-\infty}^0 \exp\big(-\frac{1}{\sigma^2}\frac{(x-y)^2}{2t}\big)dy+\int_0^\infty \exp\big(-\frac{1}{\sigma^2}\big(\frac{(x-y)^2}{2t}+y\big)\big)dy}.$$

\vspace{-0.6cm}
\begin{figure}[H]
\begin{center}
\hspace{-0.6cm}\includegraphics[width=0.5\textwidth]{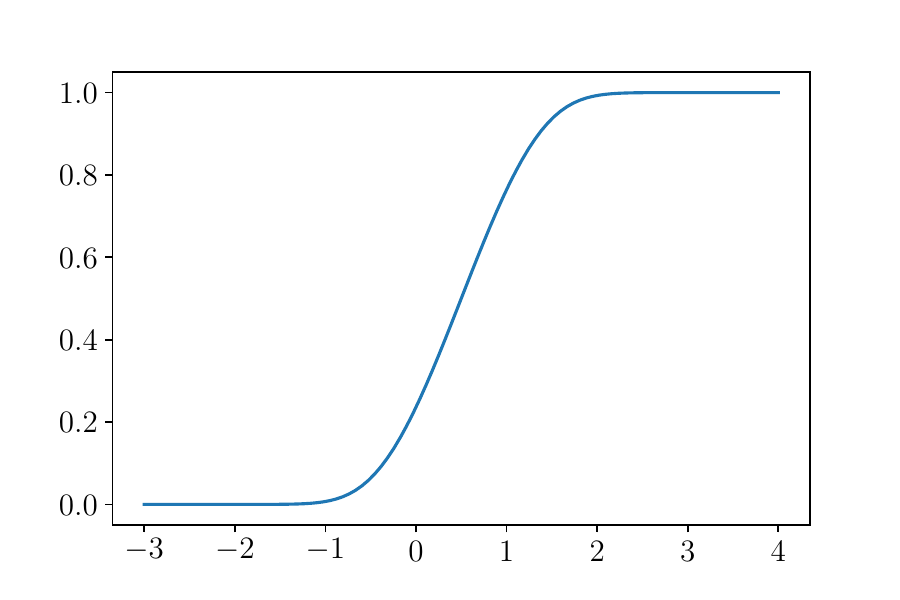}
\includegraphics[width=0.5\textwidth]{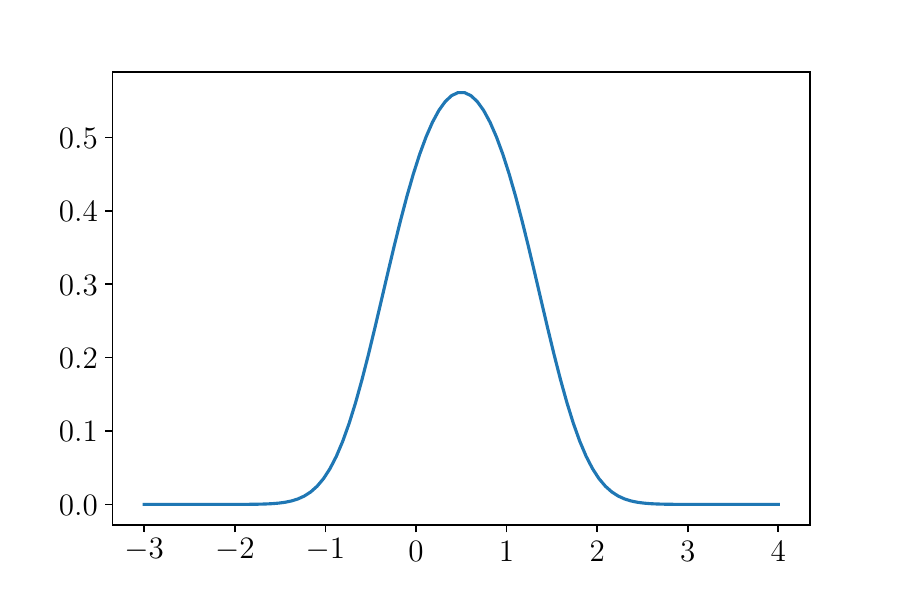}
\vspace{-0.2cm}
\caption{\label{fig: graph burgers} {\it Plots of $x \mapsto M_T(x)$ (Left) and $x \mapsto \mu_T(x)$ (Right).}}
\end{center}
\end{figure}

Figure \ref{fig: graph burgers} displays the graph of $x \mapsto M_T(x)$ and $x \mapsto \mu_T(x)$  for $T=1$. We display in Figure \ref{fig: reconstruct burgers} the reconstruction of $x \mapsto \mu_T(x)$ via $\widehat \mu_T^{N, h,  \widehat \eta^N(T,x)}(x)$ for $N = 2^{10}$ for several values of $\ell$. The value $\ell = 9$ provides with the best fit, showing again the benefit of higher-order kernels compared to a standard Gaussian kernel (with $\ell=1$). Similar results were obtained for other values of $N$. Yet, the distribution of $\widehat \eta^N(T,x)$ displayed in Figure \ref{fig: histogram bandwidth burgers} shows that our method tends to undersmooth the true density. The effect  is  less pronounced for $N=2^{15}$. In any case, our numerical results are comparable with \cite{bossy1997stochastic}.

\begin{figure}[H]
\begin{center}
\hspace{-1.8cm}\includegraphics[width=0.58\textwidth]{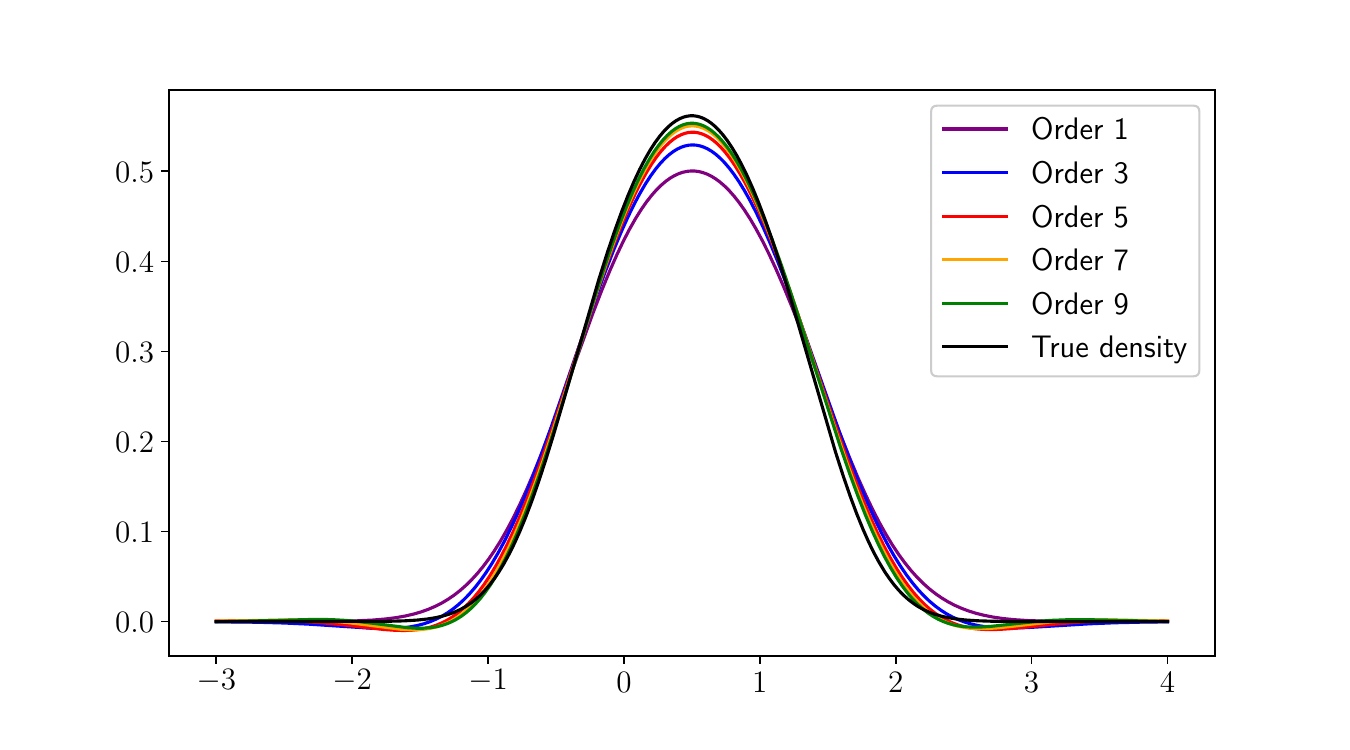} \hspace{-1.5cm}
\includegraphics[width=0.58\textwidth]{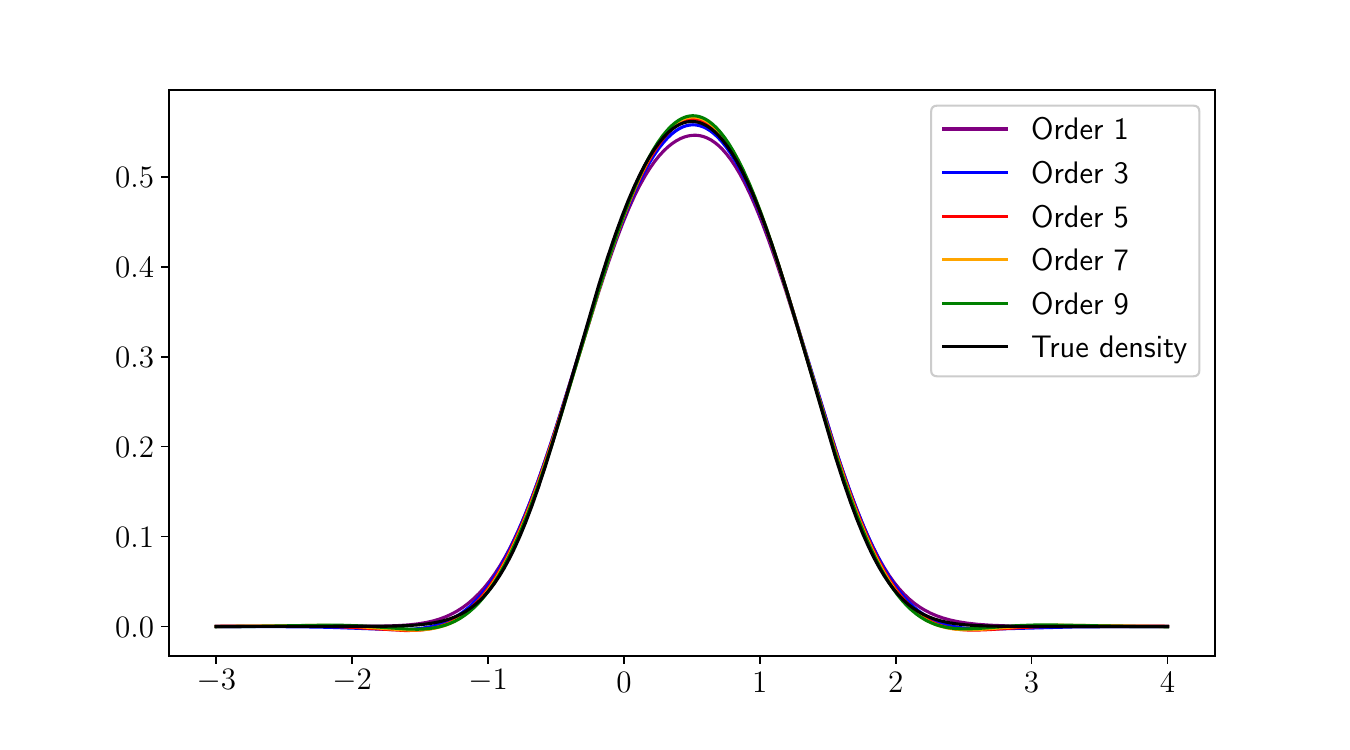} \hspace{-1.8cm}
\vspace{-0.5cm}
\caption{\label{fig: reconstruct burgers} {\it Reconstruction of $\mu_T$ by $\widehat \mu_T^{N, h,  \widehat \eta^N(T,\cdot)}$ with $\varpi = 23$, for different kernel orders for $N=2^{10}$ (Left) and $N=2^{15}$ (Right).} Higher order kernels outperform the reconstruction provided with standard Gaussian kernels ($\ell = 1$) for $N=2^{10}$.}
\end{center}
\end{figure}


\begin{figure}[H]
\begin{center}
\includegraphics[width=0.9\textwidth]{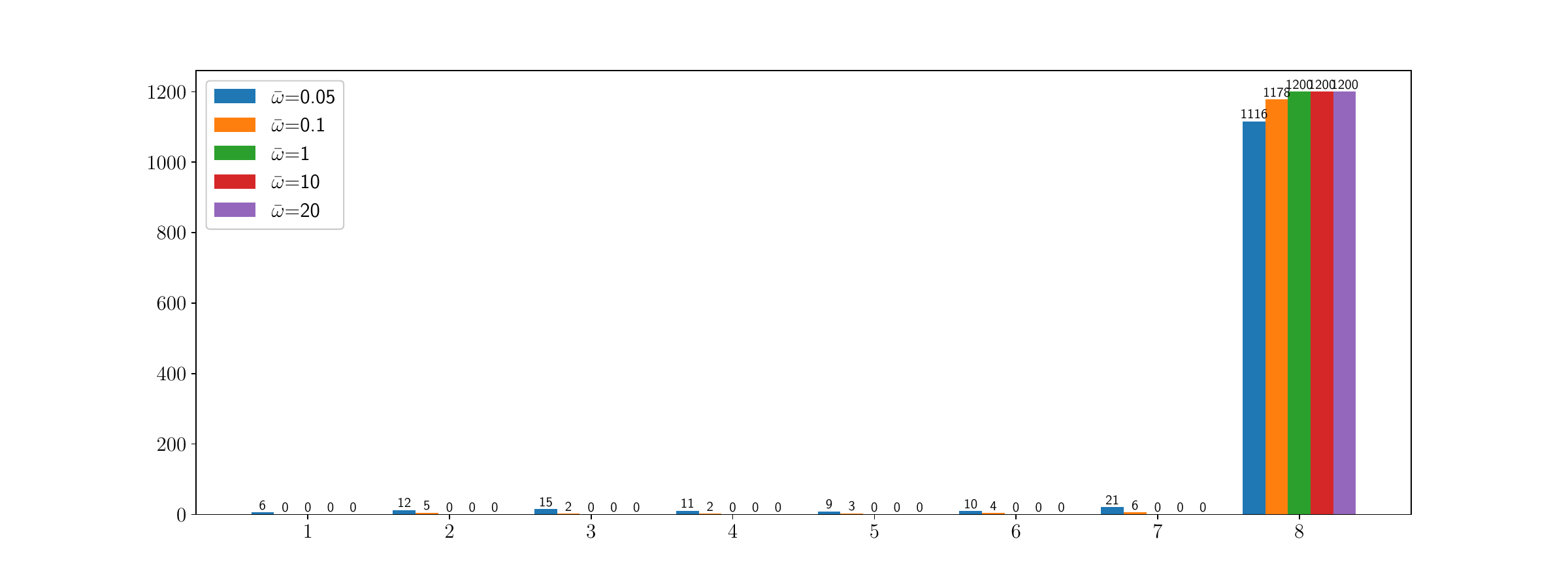} 
\vspace{-0.5cm}
\caption{\label{fig: histogram bandwidth burgers} {\it Distribution of $\widehat \eta^N(T,x)$}. The domain $x \in [-3,4]$ is computed for 1200 samples over a discrete grid of $100$ points, {\it i.e.} mesh $8 \cdot 10^{-2}$, $N = 2^5, 2^6, \ldots, 2^{16}$ for $\ell = 7$. The method tends pick the largest bandwidth.}
\end{center}
\end{figure}

\section{Proofs} \label{sec: proofs}

\subsection{Proof of Theorem \ref{thm: deviation}}
The proof of Theorem \ref{thm: deviation} relies on Proposition~\ref{prop: bernstein} below, which establishes a Bernstein inequality for the fluctuations of $\bar \mu_T^{N,h}(dx)-\bar \mu_T^h(dx)$.   Recall that $\bar \mu_T^{N,h}$ denotes the empirical distribution of the interacting particle system at time $T$ generated by  \eqref{discresys}, while $\bar \mu_T^h$ denotes the law of the random variable defined by the abstract Euler scheme (see Definition \ref{def:abstract-euler-scheme}).  
\begin{prop} \label{prop: bernstein}
Work under Assumptions \ref{ass: AI}, \ref{ass: AII} and \ref{ass: AIII 1}. Let $N\geq 2$ and $Nh\leq \mathfrak C$. For any real-valued bounded function $\varphi$ defined on $\R^d$ and any $\varepsilon \geq 0$, we have:
\begin{equation} \label{eq: main Bernstein}
\PP\Big( \int_{\R^d}\varphi \,d(\bar \mu_T^{N,h}-\bar \mu_T^h) \geq \varepsilon \Big) \leq \kappa_1 \exp\Big(-\frac{\kappa_2N\varepsilon^2}{|\varphi|^2_{L^2(\mu_T)}+|\varphi|_\infty\varepsilon}\Big),
\end{equation}
where $\kappa_1$ depends on $\mathfrak{C}$ and $\Cmodel$, and $\kappa_2$ only on $\Cmodel$, both specified explicitly in the proof. 
\end{prop}

The proof of  Proposition \ref{prop: bernstein}  follows the strategy of Theorem 18 in \cite{della2021nonparametric}. We repeat the main steps in order to remain self-contained and highlight the important modifications we need in the context of Euler scheme approximations.  A key ingredient of the proof of Proposition \ref{prop: bernstein} is the comparison between the interacting particle system defined in \eqref{discresys} and an auxiliary system 
$(Y_{t}^{1,h},\ldots,Y_{t}^{N,h})_{t\in[0,T]}$ consisting of $N$ independent copies of the abstract Euler scheme \eqref{eq: euler abstract} (i.e., without interaction):  
\begin{equation}\label{sysY}
\begin{cases}
dY_{t}^{n,h}=b(\underline{t}, Y_{\underline{t}}^{n,h}, \mu_{\underline{t}})dt + \sigma(\underline{t}, Y_{\underline{t}}^{n,h})dB_{t}^{n}, \quad 1 \leq n \leq N,\\
(Y_{0}^{1, h}, \ldots , Y_{0}^{N, h}) = (\bar{X}_{0}^{1, h}, \ldots , \bar{X}_{0}^{N, h}),\\
\end{cases}
\end{equation}
For $t\in[0, T]$, let
\begin{equation}\label{eq:def-process-L}
 L^{N,h}_t := \sum_{n = 1}^N \int_{0}^{t}\big((c^{-1/2} b)(\underline{s}, Y_{\underline{s}}^{n,h}, \tfrac{1}{N}\sum_{n=1}^{N}\,\delta_{Y^{n,h}_{\underline{s}}})-(c^{-1/2}b)(\underline{s}, Y_{\underline{s}}^{n,h}, \mu_{\underline{s}})\big)^\top dB_{s}^{n},   
\end{equation}
where $c^{-1/2}$ is any square root of $c^{-1}=(\sigma \sigma^\top)^{-1}$, which is well defined by the ellipticity assumption \ref{ass: AII}. 
The next lemma provides a key estimate that enables a change of probability argument in the proof of Proposition~\ref{prop: bernstein}. 
Its proof is postponed to the end of this section.
\begin{lemma}\label{expp}
Work under Assumptions \ref{ass: AI}, \ref{ass: AII} and \ref{ass: AIII 1}. Assume that $Nh \leq  \mathfrak C$ for some constant $\mathfrak C>0$. For every $\kappa>0$, there exists $\delta(\kappa)>0$ such that 
\begin{equation}\label{newlem4-2}
\forall\,\delta\in [0, \delta(\kappa)],  \quad\quad\,\sup_{t\in[0, T-\delta(\kappa)]}\mathbb{E}_{\mathbb{P}}\Big[\exp\Big(\kappa\big(\langle L_{\cdot}^{N,h}\rangle_{t+\delta}-\langle L_{\cdot}^{N,h}\rangle_{t}\big)\Big)\Big]\leq  C' 
\end{equation}
where $C' >0$ only depends on $\mathfrak C$  and $\Cmodel$. 
\end{lemma}

\begin{proof}[Proof of Proposition \ref{prop: bernstein}]
We work on a rich enough filtered probability space  $(\Omega, \mathcal{F}, (\mathcal{F}_{t})_{t\geq0}, \mathbb{P})$ in order to accomodate all the random quantities needed.
First, note that the abstract Euler scheme $(\bar{X}^h_{t})_{t\in[0, T]}$ defined by  \eqref{eq: euler abstract} solves
\begin{equation}\label{eqeuler}
d\bar{X}^h_{t}=b(\underline{t}, \bar{X}^h_{\underline{t}}, {\mu}_{\underline{t}})dt+\sigma(\underline{t}, \bar{X}^h_{\underline{t}}) dB_t, \quad \mathcal L(\bar{X}_{0})=\mu_0,
\end{equation}
where $\underline{t}$ is defined by 
\[\forall\;  0 \leq m\leq M \text{ and } t\in[t_m, t_{m+1}), \quad \underline{t}\coloneqq t_{m}.\]
Similarly,  $(\bar{X}_{t}^{1,h}, ..., \bar{X}_{t}^{N,h})_{t\in[0, T]}$ defined by \eqref{discresys} solves
\begin{equation}\label{eqpart}
\begin{cases}
d\bar{X}^{n, h}_{t}=b(\underline{t}, \bar{X}^{n, h}_{\underline{t}}, \bar{\mu}^{N,h}_{\underline{t}})dt+\sigma(\underline{t}, \bar{X}^{n, h}_{\underline{t}}) dB^{n}_t,\quad 1\leq n\leq N,\\
\bar{\mu}_{\underline{t}}^{N,h}=\frac{1}{N}\sum_{n=1}^{N}\delta_{\bar{X}^{n, h}_{\underline{t}}},\\
\mathcal L(\bar{X}_{0}^{1, h}, \ldots , \bar{X}_{0}^{N, h}) = \mu_0^{\otimes N}.
\end{cases}
\end{equation}
\textsc{Step 1.}  Consider the particle system $(Y_{t}^{1,h},\ldots,Y_{t}^{N,h})_{t\in[0,T]}$ defined in \eqref{sysY} and the process $(L^{N,h}_t)_{t\in[0,T]}$ introduced in \eqref{eq:def-process-L}.
For $t\in[0,T]$, define
\[
\mathcal{E}_{t}(L_{\cdot}^{N,h})\coloneqq\exp\big(L^{N,h}_{t}-\frac{1}{2}\langle L^{N,h}_{\cdot}\rangle_{t}\big),
\]
where $\langle L^{N,h}_{\cdot} \rangle_t$ denotes the predictable compensator of $L^{N,h}_{t}$. Applying Lemma~\ref{expp} with $\kappa=\tfrac{1}{2}$, together with Novikov’s criterion (see, e.g., \cite[Proposition~3.5.12, Corollary~3.5.13, and 3.5.14]{Karatzas1991Brownian}), we deduce that $(\mathcal{E}_{t}(L^{N,h}_{\cdot}))_{t\in[0,T]}$ is an $(\mathcal{F}_t,\mathbb{P})$-martingale as soon as $Nh\leq \mathfrak C$.
This allows us to define a new probability distribution $\mathbb{Q}$ on $(\Omega,\mathcal{F}_T)$ by
\begin{equation}\label{mesureq}
\mathbb{Q}\coloneqq\mathcal{E}_{T}(L^{N,h}_{\cdot})\cdot \mathbb{P}.
\end{equation}
By Girsanov's Theorem (see e.g. \cite[Theorem 3.5.1]{Karatzas1991Brownian}), 
we have
\begin{equation}\label{relation}
\mathbb{Q}\circ (Y_{t}^{1,h}, ..., Y_{t}^{N,h})_{t\in[0, T]}^{-1}=\mathbb{P}\circ (X_{t}^{1, h}, ..., X_{t}^{N,h})_{t\in[0, T]}^{-1}. 
\end{equation}
\textsc{Step 2.} We claim that for any subdivision $0=v_0 < v_1 < ... < v_K=T$ and for any $\mathcal{F}_T$-measurable event $A^{N}$, we have
\begin{equation}\label{ineqstep1}
\EE_{\PP}\Big[\mathbb{Q}\big( A^{N}\;\big|\; \mathcal{F}_{v_{j-1}}\big)\Big]\leq \EE_{\PP}\Big[\mathbb{Q}\big( A^{N}\;\big|\; \mathcal{F}_{v_{j}}\big)\Big]^{\frac{1}{4}} \EE_{\PP}\Big[ \exp \Big( 2 \big( \langle L_{\cdot}^{N,h}\rangle_{v_j}- \langle L_{\cdot}^{N,h}\rangle_{v_{j-1}} \big)\Big)\Big]^{\frac{1}{4}},\: 1\leq j\leq K.
\end{equation}
The proof is the same as in Step 3 of the proof of Theorem 18 in \cite{della2021nonparametric} and  is inspired from the estimate (4.2) in Theorem 2.6 in \cite{lacker2018on}. We repeat the argument: we have 
\begin{align*}\label{inq}
\EE_{\PP}&\big[\mathbb{Q}\big( A^{N}\;\big|\; \mathcal{F}_{v_{j-1}}\big)\big]=\EE_{\PP}\big[ \EE_{\QQ}\big[\mathbb{Q}\big( A^{N}\;\big|\; \mathcal{F}_{v_{j}}\big)\;\big|\; \mathcal{F}_{v_{j-1}}\big]\big]\nonumber\\
&= \EE_{\PP}\Big[ \EE_{\PP}\big[\,\frac{\,\mathcal{E}_{v_{j}}(L_{\cdot}^{N,h})\,}{\,\mathcal{E}_{v_{j-1}}(L_{\cdot}^{N,h})\,}\,\mathbb{Q}\big( A^{N}\;\big|\; \mathcal{F}_{v_{j}}\big)\;\big|\; \mathcal{F}_{v_{j-1}}\big]\Big]= \EE_{\PP}\Big[\,\frac{\,\mathcal{E}_{v_{j}}(L_{\cdot}^{N,h})\,}{\,\mathcal{E}_{v_{j-1}}(L_{\cdot}^{N,h})\,}\,\mathbb{Q}\big( A^{N}\;\big|\; \mathcal{F}_{v_{j}}\big)\Big],
\end{align*}
where the second inequality follows from Bayes's rule in \cite[Lemma 3.5.3]{Karatzas1991Brownian}. 
Next, we have
\begin{align}
\frac{\,\mathcal{E}_{v_{j}}(L_{\cdot}^{N,h})\,}{\,\mathcal{E}_{v_{j-1}}(L_{\cdot}^{N,h})\,}=\mathcal{E}_{v_j}\big(2(L_{\cdot}^{N,h}-L_{v_{j-1}}^{N,h})\big)^{\frac{1}{2}}\cdot \exp\big(\langle L_{\cdot}^{N,h}\rangle_{v_{j}}-\langle L_{\cdot}^{N,h}\rangle_{v_{j-1}}\big)^{\frac{1}{2}}. \nonumber
\end{align}
The process $\big(\mathcal{E}_{v_j}\big(2(L_{\cdot}^{N,h}-L_{v_{j-1}}^{N,h})\big)\big)_{t\in[v_{j-1}, T]}$ is a $(\mathcal F_t, \PP)$ martingale if $Nh\leq \mathfrak C$.  Hence  
$$\EE_{\PP}\big[\mathcal{E}_{v_j}\big(2(L_{\cdot}^{N,h}-L_{v_{j-1}}^{N,h})\big)\big]=1$$
for every $t \in [v_{j-1}, T]$. It follows that
\begin{align*}
\EE_{\PP}&\Big[\mathbb{Q}\big( A^{N}\;\big|\; \mathcal{F}_{v_{j-1}}\big)\Big]\nonumber\\
&=\EE_{\PP}\left[\,\mathcal{E}_{v_j}\big(2(L_{\cdot}^{N,h}-L_{v_{j-1}}^{N,h})\big)^{\frac{1}{2}}\cdot \exp\big(\langle L_{\cdot}^{N,h}\rangle_{v_{j}}-\langle L_{\cdot}^{N,h}\rangle_{v_{j-1}}\big)^{\frac{1}{2}}\,\mathbb{Q}\big( A^{N}\;\big|\; \mathcal{F}_{v_{j}}\big)\right]\nonumber\\
&\leq  \Big(\EE_{\PP}\left[\, \exp\big(\langle L_{\cdot}^{N,h}\rangle_{v_{j}}-\langle L_{\cdot}^{N,h}\rangle_{v_{j-1}}\big)\,\mathbb{Q}\big( A^{N}\;\big|\;  \mathcal{F}_{v_{j}}\big)^{2}\right]\Big)^{\frac{1}{2}}\nonumber\\
&\leq  \EE_{\PP}\left[\, \exp\Big(2\big(\langle L_{\cdot}^{N,h}\rangle_{v_{j}}-\langle L_{\cdot}^{N,h}\rangle_{v_{j-1}}\big)\Big)\right]^{\frac{1}{4}}\EE_{\PP}\left[\,\mathbb{Q}\big( A^{N}\;\big|\; \mathcal{F}_{v_{j}}\big)^{4}\right]^{\frac{1}{4}}\nonumber\\
&\leq \EE_{\PP}\left[\, \exp\Big(2\big(\langle L_{\cdot}^{N,h}\rangle_{v_{j}}-\langle L_{\cdot}^{N,h}\rangle_{v_{j-1}}\big)\Big)\right]^{\frac{1}{4}}\EE_{\PP}\left[\,\mathbb{Q}\big( A^{N}\;\big|\; \mathcal{F}_{v_{j}}\big)\right]^{\frac{1}{4}},
\end{align*} 
where the first two inequalities follows from Cauchy-Schwarz's inequality and the last inequality follows from Jensen's inequality. Thus \eqref{ineqstep1} is established.

\smallskip

\noindent\textsc{Step 3.} 
Let $A^{N}\in \mathcal{F}_{T}$. Since $\big(\mathcal{E}_{t}(L^{N,h}_{\cdot})\big)_{t\in[0, T]}$ is a $(\mathcal F_t,\PP)$ martingale,  $\PP$ and $\QQ$ coincide on $\mathcal{F}_{0}$ \big(see e.g. \cite[Section 3 - (5.5)]{Karatzas1991Brownian}\big). It follows that 
\[\QQ (A^{N})=\EE_{\QQ}\big[\,\QQ ( A^{N}\,\big|\,\mathcal{F}_0)\,\big]=\EE_{\PP}\big[\,\QQ ( A^{N}\,\big|\,\mathcal{F}_0)\,\big].\]
Now, let $m \geq 1$ and take a subdivision $0=v_0<v_1< \ldots <v_m=T$ such that
\[ \forall \,0\leq k\leq m-1, \quad v_{k+1}-v_{k}\leq \delta(2)\] where $\delta(2)$ is the constant in Lemma \ref{expp} for $\kappa=2$. It follows by \eqref{ineqstep1} that 
\begin{align}
\EE_{\PP}&\big[\,\QQ ( A^{N}\,\big|\,\mathcal{F}_0)\,\big]\leq \EE_{\PP}\big[\mathbb{Q}\big( A^{N}\;\big|\; \mathcal{F}_{T}\big)\big]^{4^{-m}}\prod_{j=1}^{m} \EE_{\PP}\Big[ \exp \Big( 2 \big( \langle L_{\cdot}^{N,h}\rangle_{v_j}- \langle L_{\cdot}^{N,h}\rangle_{v_{j-1}} \big)\Big)\Big]^{\frac{ j}{4}}\nonumber\\
&\leq \PP (A^{N})^{4^{-m}}\prod_{j=1}^{m} \EE_{\PP}\Big[ \exp \Big( 2 \big( \langle L_{\cdot}^{N,h}\rangle_{v_j}- \langle L_{\cdot}^{N,h}\rangle_{v_{j-1}} \big)\Big)\Big]^{\frac{ j}{4}}\quad\quad \text{(since $A^{N}\in \mathcal{F}_{T}$)}\nonumber\\
&\leq \PP (A^{N})^{4^{-m}} \Big(\sup_{t\in[0, T-\delta(2)]}\mathbb{E}_{\mathbb{P}}\Big[\exp\Big(2\big(\langle L_{\cdot}^{N,h}\rangle_{t+\varepsilon}-\langle L_{\cdot}^{N,h}\rangle_{t}\big)\Big)\Big]\Big)^{\frac{m(m+1)}{8}}\nonumber\\
&\leq \PP (A^{N})^{4^{-m}} C'^{\frac{m(m+1)}{8}}
\end{align}
by applying Lemma \ref{expp} with $\kappa=2$.

\smallskip

\noindent\textsc{Step 4.} We first recall the Bernstein's inequality: if $Z_1, \ldots , Z_N$ are real-valued  centered  independent random variables bounded by some constant $Q$, we have 
\begin{equation}\label{bernsteinclassical}
\forall\, \varepsilon\geq 0,\quad\quad \PP\Big( \sum_{n=1}^{N}Z_n\geq \varepsilon\Big)\leq \exp\Big(\;-\;\frac{\varepsilon^{2}}{\,2\sum_{n=1}^{N} \EE[Z_n^2]+\frac{2}{3}Q\varepsilon}\Big).
\end{equation}
Now, for $\varepsilon \geq 0$, we pick
\begin{align*}
A^{N}_{T}\coloneqq&  \Big\{ \sum_{n=1}^{N}\Big( \varphi\big(Y_{T}^{n,h}\big)-\int_{\RR^d}\varphi(y)\bar{\mu}_{T}^h(dy)\Big)\geq N\varepsilon\Big\}
\end{align*}
so that $A^{N}_{T}\in \mathcal{F}_{T}$.  Since  $(Y_T^{1,h}, \ldots , Y_T^{N,h})$ are independent with common distribution $\bar \mu_T^h$, we have
\begin{equation}
\mathbb P (A^{ N}_{T})\leq \exp\Big(\;-\;\frac{N\varepsilon^{2}}{\,2|\varphi| _{L^{2}(\bar{\mu}_T^h)}^2+\frac{2}{3}|\varphi|_{\infty}\varepsilon\,}\Big).
\end{equation}
by \eqref{bernsteinclassical}  with $Z_{n} = \varphi(Y_T^{n,h})- \int_{\R^d}\varphi(x)\bar \mu_T^h(dx)$, having $\EE[Z_{n}^2] \leq |\varphi|_{L^2(\bar \mu_T^h)}^2$ and $Q=|\varphi|_\infty$. It follows by \eqref{relation} that 
\[
\PP\Big(\int_{\RR^d}\varphi \,d(\bar{\mu}_{T}^{N,h}-\bar{\mu}_{T}^h)\geq \varepsilon \Big)=\PP\Big( \sum_{n=1}^{N}\Big( \varphi\big(\bar{X}_{T}^{n, h}\big)-\int_{\RR}\varphi(y)\bar{\mu}_{T}^h(dy)\Big)\geq N\varepsilon \Big)=\QQ (A^{N}_{T}).\]
By \textsc{Step 3}, we infer
\begin{align*}
\QQ (A^{N}_{T}) & \leq \PP (A^{N}_{T})^{4^{-m}}C'^{\frac{m(m+1)}{8}}  \leq  C'^{\frac{m(m+1)}{8}} \exp\Big(-4^{-m}\frac{N\varepsilon^{2}}{2|\varphi|_{L^{2}(\bar{\mu}_T^h)}^2+\frac{2}{3}|\varphi|_{\infty}\varepsilon\,}\Big), 
\end{align*}
and we conclude by taking $\kappa_1=C'^{\frac{m(m+1)}{8}}$ and $\kappa_2=2^{-1}4^{-m}$. 
\end{proof}

\subsubsection*{Completion of proof of Theorem \ref{thm: deviation}}
Recall the notation $K_\eta(x) = \eta^{-d}K(\eta^{-1}x)$ and set $\varphi \star \psi(x) = \int_{\R^d} \varphi(x-y)\psi(y)dy$ for the convolution product between two integrable functions $\varphi,\psi:\R^d \rightarrow \R$.
We have
\begin{align}
& \PP\big(\big|\widehat \mu_{T}^{N,h,\eta}(x)-\mu_{T}(x)\big| \geq \varepsilon \big) \nonumber \\
&\leq \PP\big(\big|\widehat \mu^{N,h,\eta}_{T}(x)-K_\eta \star \bar{\mu}^{h}_{T}(x)\big| \geq \varepsilon/3 \big) + {\bf 1}_{\{|K_\eta \star \bar \mu^h_{T}(x)-K_\eta \star \mu_{T}(x)| \geq \varepsilon/3\}}+{\bf 1}_{\{|K_\eta \star \mu_{T}(x)-\mu_{T}(x)| \geq \varepsilon/3\}} \label{eq: union bound argument}.
\end{align}
We have
$$\big|K_\eta \star \bar \mu^{h}_{T}(x)-K_\eta \star \mu_{T}(x)\big| \leq |K|_{L^1}\sup_{x \in \RD}\big|\bar \mu_{T}^h(x)-\mu_{T}(x)\big| $$
since $|K_\eta|_{L^1}  = |K|_{L^1}$ and this last term is strictly smaller than $\varepsilon/3$ for $h < h^\star\big( |K|_{L^1}^{-1}\tfrac{1}{3}\varepsilon\big)$, 
therefore the second indicator is $0$. Likewise
$$\big|K_\eta \star \mu_{T}(x)-\mu_{T}(x)\big| \leq \mathcal B_\eta(\mu_{T}, x) < \varepsilon/3$$
for $\eta < \eta^\star(\tfrac{1}{3}\varepsilon,  \mu_T)$ and the third  indicator is $0$ as well. Finally
\begin{align*}
&\PP\Big(\big|\widehat \mu^{N,h,\eta}_{T}(x)-K_\eta \star \bar \mu^{h}_{T}(x)\big| \geq \tfrac{1}{3}\varepsilon \Big) \nonumber\\
&=\PP\Big(\big|\int_{\R^d} K_\eta(x-\cdot)d(\bar \mu^{N,h}_{T}-\bar \mu^h_{T})\big| \geq \tfrac{1}{3}\varepsilon \Big)
\leq 2\kappa_1 \exp\Big(-\frac{\kappa_2N(\tfrac{1}{3}\varepsilon)^2}{|K_\eta(x-\cdot)|^2_{L^2(\mu_{T})}+|K_\eta|_\infty\varepsilon}\Big)  
\end{align*}
by Proposition \ref{prop: bernstein}. Theorem \ref{thm: deviation} then follows, where we reuse $\kappa_2$ (instead of $\kappa_2/9$) for simplicity, as constants may be redefined without loss of generality. 

\subsubsection*{Proof of Lemma \ref{expp}}

Writing  $\widetilde \mu^{N,h}_s = \tfrac{1}{N}\sum_{n=1}^{N}\,\delta_{Y^{n,h}_{s}}$, for $\kappa >0$, we have
\begin{align} 
& \kappa \big(\langle L_{\cdot}^{N,h}\rangle_{t+\delta}-\langle L_{\cdot}^{N,h}\rangle_{t}\big) \nonumber\\
& = \kappa \sum_{n = 1}^N \int_{t}^{t+\delta}\big(b(\underline{s}, Y_{\underline{s}}^{n,h},\widetilde \mu^{N,h}_{\underline{s}})-b(\underline{s}, Y_{\underline{s}}^{n,h}, \mu_{\underline{s}})\big)^\top c(\underline{s}, Y_{\underline{s}}^{n,h})^{-1} \big(b(\underline{s}, Y_{\underline{s}}^{n,h}, \widetilde \mu^{N,h}_{\underline{s}})-b(\underline{s}, Y_{\underline{s}}^{n,h}, \mu_{\underline{s}})\big)ds \nonumber\\
& \leq \kappa \sup_{t\in [0,T]}|\mathrm{Tr}(c(t,\cdot)^{-1})|_\infty \sum_{n = 1}^N\int_{t}^{t+\delta}\big|b(\underline{s}, Y_{\underline{s}}^{n,h},\widetilde \mu^{N,h}_{\underline{s}})-b(\underline{s}, Y_{\underline{s}}^{n,h}, \mu_{\underline{s}})\big|^2ds \nonumber\\
& \leq  2\kappa \sup_{t\in [0,T]}|\mathrm{Tr}(c(t,\cdot)^{-1})|_\infty  \Big( \sum_{n = 1}^N\int_{t}^{t+\delta}\big|b(\underline{s}, Y_{\underline{s}}^{n,h},\widetilde \mu^{N,h}_{\underline{s}})-b(\underline{s}, Y_{\underline{s}}^{n,h}, \bar \mu_{\underline{s}}^h)\big|^2ds \nonumber\\
&\hspace{6cm}+N |\widetilde b|_{\mathrm{Lip}}^2\int_{t}^{t+\delta} \mathcal W_1(\bar \mu_{\underline{s}}^h, \mu_{\underline{s}})^2 ds\Big),  \label{eq: key estimate}
\end{align}
by the uniform ellipticity of $c$ thanks to Assumption \ref{ass: AII} and by triangle inequality for the last inequality, with $|\widetilde b(t,x,y)- \widetilde b(t,x,y')\big| \leq  |\widetilde b|_{\mathrm{Lip}}|y-y'|$ in particular.  Recall that under Assumption \ref{ass: AI}, \ref{ass: AII} and \ref{ass: AIII 1}, there exists a unique strong solution $X=(X_t)_{t\in[0,T]}$ of the McKean-Vlasov equation \eqref{mckeaneq} satisfying 
\[\forall\,s,t\in[0,T], s<t,\qquad \E\big[|X_t - X_s|^2\big]^{1/2}\leq  \Cmodel \sqrt{t-s},\]
see \cite[Proposition 2.1]{liu2020functional}. 
Hence, the function $(t,x)\mapsto b^{\,\mu}(t, x)\coloneqq b(t, x, \mu_t)$ with $\mu_t=\mathcal{L}(X_t)$ is $\frac{1}{2}$-H\"older continuous in $t$ and Lipschitz continuous in $x$. Thus, one can obtain
$$\sup_{s \in [0,T]}\mathcal W_1(\bar \mu_{\underline{s}}^h, \mu_{\underline{s}}) \leq  \Cmodel\, h^{1/2},$$
by rewriting \eqref{mckeaneq} as a Brownian diffusion \[dX_t=b^{\,\mu}(t, X_t)dt+\sigma(t,X_t)dB_t\] and by applying the classical convergence result of the Euler scheme for a Brownian diffusion (see e.g. \cite[Theorem 7.2]{pages2018numerical}) since for every $s\in[0,T]$, $\mathcal W_1(\bar \mu_{s}^h, \mu_{s})\leq \E\big[|X_s- \bar{X}_s^h|\big]$. 
Writing
\begin{align*}
b(\underline{s}, Y_{\underline{s}}^{n,h},\widetilde \mu^{N,h}_{\underline{s}})-b(\underline{s}, Y_{\underline{s}}^{n,h}, \bar \mu_{\underline{s}}^h) & = N^{-1}\sum_{n'=1}^N\xi_s^{n,n'} = N^{-1}\xi_s^{n,n}+\tfrac{N-1}{N} S_s^{n,N},
\end{align*} 
with $S_s^{n,N} = \tfrac{1}{N-1}\sum_{n' \neq n}\xi_s^{n,n'}$ and 
$$\xi_s^{n,n'} = \widetilde b(\underline{s}, Y_{\underline{s}}^{n,h},Y_{\underline{s}}^{n',h})-\int_{\R^d}\widetilde b(\underline{s}, Y_{\underline{s}}^{n,h}, y)\bar \mu_{\underline{s}}^h(dy).$$ 
The random variables $(\xi_s^{n,n'})_{1 \leq n' \leq N, n'\neq n}$ are centered, identically distributed and conditionally independent given $Y_{\underline{s}}^{n,h}$. Moreover, for every integer $m \geq 1$, we have the estimate
\begin{align}\label{eq:moment-Y-2m}
\sup_{s \in [0,T]}\E\big[|Y_{\underline{s}}^{n,h}|^{2m}\big] \leq  \Cmodel^m m!.
\end{align}

The proof  of \eqref{eq:moment-Y-2m}  is standard (see for instance, \cite{meleard1996asymptotic, sznitman1991topics} or \cite{della2021nonparametric} for a control of growth in the constant $m$ and we omit it).
From \eqref{eq:moment-Y-2m}, we infer, for every $m \geq 1$:
\begin{align*}
&\E\Big[\big|\xi_s^{n,n'}\big|^{2m}\,|\,Y_{\underline{s}}^{n,h}\Big] = \EE \Big[\; \big| \widetilde{b}(\underline{s}, Y_{\underline{s}}^{n, h}, Y_{\underline{s}}^{n', h})- \int_{\RD} \widetilde{b}(\underline{s}, Y_{\underline{s}}^{n, h}, y)\bar{\mu}_{\underline{s}}^{h}(dy)\big|^{2m}\;\Big|\; Y_{\underline{s}}^{n,h}\; \Big]\nonumber\\
&\quad\leq \EE \Big[\; 2^{2m-1}\Big(\big| \widetilde{b}(\underline{s}, Y_{\underline{s}}^{n, h}, Y_{\underline{s}}^{n', h})- \widetilde{b}(\underline{s}, Y_{\underline{s}}^{n, h}, 0)\big|^{2m} \nonumber\\
&\quad\qquad +\big|\widetilde{b}(\underline{s}, Y_{\underline{s}}^{n, h}, 0)-\int_{\RD} \widetilde{b}(\underline{s}, Y_{\underline{s}}^{n, h}, y)\bar{\mu}_{\underline{s}}^{h}(dy)\big|^{2m}\Big)\;\Big|\; Y_{\underline{s}}^{n,h}\; \Big]\nonumber\\
&\quad \leq \EE\Big[\; 2^{2m-1}\Big(\; |\widetilde{b}|_{\mathrm{Lip}}^{2m} \,|Y_{\underline{s}}^{n', h}|^{2m}+ |\widetilde{b}|_{\mathrm{Lip}}^{2m}\, \EE\big[  |Y_{\underline{s}}^{n', h}|^{2m}\big] \;\Big)\;\Big|\; Y_{\underline{s}}^{n,h}\; \Big] =2^{2m}\,|\widetilde{b}|_{\mathrm{Lip}}^{2m} \,\EE\big[  |Y_{\underline{s}}^{n', h}|^{2m}\big] \nonumber\\
&\quad \leq \big(4 |\widetilde{b}|_{\mathrm{Lip}}^{2}  \Cmodel \big)^m  m!
\end{align*} 
and in turn,  we derive, for every $v \in \R$,  
\[ \E\big[\exp(v(\xi_s^{n,n'})_{[k]})\,|\,Y_{\underline{s}}^{n,h}\big] \leq \exp\big(\tfrac{ \Cmodel  v^2}{2}\big),\;\;k=1,\ldots,d.\]
In other words, conditional on $Y_{\underline{s}}^{n,h}$, each component $(\xi_s^{n,n'})_{[k]}$  of $\xi_s^{n,n'} = \big((\xi_s^{n,n'})_{[1]}, \ldots, (\xi_s^{n,n'})_{[d]}\big)$ is sub-Gaussian with variance proxy  depending only on the model constant $\Cmodel$, which we continue to denote by $\Cmodel$ in the subsequent computations   (see e.g. \cite{MR598605} and \cite[Theorem 2.1.1]{pauwels2020lecture}). Consequently, conditional on $Y_{\underline{s}}^{n,h}$,  each component  $(S_s^{n,N})_{[k]}$  of $S_s^{n,N}$ is sub-Gaussian with variance proxy $(N-1)^{-1}  \Cmodel $, by the conditional independence of the random variables $(\xi_s^{n,n'})_{1 \leq n' \leq N, n'\neq n}$.
This implies in particular
$$
\E\Big[\exp\Big(\frac{(N-1)(S_s^{n,N})_{[k]}^2}{8 \Cmodel}\Big)\Big]  =\E\Big[\,\E\Big[\exp\Big(\frac{(N-1)(S_s^{n,N})_{[k]}^2}{8 \Cmodel}\Big)\;\Big|\; Y_{\underline{s}}^{n,h}\Big]\,\Big] \leq 2,
$$
and in turn
\begin{equation}  \label{eq: the subgaussian ineq}
\E\Big[\exp\Big(\frac{(N-1)|S_s^{n,N}|^2}{8d \Cmodel}\Big)\Big] \leq d^{-1}\sum_{k = 1}^d\E\Big[\exp\Big(\frac{(N-1)(S_s^{n,N})_{[k]}^2}{8 \Cmodel}\Big)\Big] \leq 2.
\end{equation}
Likewise
\begin{equation}  \label{eq: the subgaussian ineq bis}
\E\Big[\exp\Big(\frac{|\xi_s^{n,n}|^2}{8d \Cmodel}\Big)\Big] \leq 2.
\end{equation}
Abbreviating $\tau = 2\kappa  \sup_{t\in [0,T]}|\mathrm{Tr}(c(t,\cdot)^{-1})|_\infty $, it follows that
\begin{align*}
&\E\Big[\exp\big(\kappa (\langle L_{\cdot}^{N,h}\rangle_{t+\delta}-\langle L_{\cdot}^{N,h}\rangle_{t})\big)\Big] \\
& \leq e^{\tau|\widetilde b|_{\mathrm{Lip}}^2 \Cmodel^2\delta Nh}
 \E\Big[\exp\Big(\tau\int_{t}^{t+\delta}\sum_{n = 1}^N\big|N^{-1}\sum_{n'=1}^N \xi_s^{n,n'}\big|^2ds\Big)\Big] \\
 & \leq  e^{\tau|\widetilde b|_{\mathrm{Lip}}^2 \Cmodel^2\delta Nh} \tfrac{1}{2\delta}\int_{t}^{t+\delta}N^{-1}\sum_{n = 1}^N\Big(\E\Big[\exp\big(4\delta\tau N\big|N^{-1}\xi_s^{n,n}\big|^2\big)\Big]+\E\Big[\exp\big(4\tau \delta N\big|\tfrac{N-1}{N}S_s^{n,N}\big|^2\big)\Big]\Big)ds \\
& \leq 2 e^{\tau|\widetilde b|_{\mathrm{Lip}}^2 \Cmodel^2\delta Nh} 
\end{align*}
as soon as $\delta \leq (32\tau d\Cmodel )^{-1} \tfrac{N}{N-1} = (64 \kappa \sup_{t\in [0,T]}|\mathrm{Tr}(c(t,\cdot)^{-1})|_\infty  d\Cmodel )^{-1} \tfrac{N}{N-1}$ by \eqref{eq: the subgaussian ineq} and \eqref{eq: the subgaussian ineq bis}. We obtain Lemma \ref{expp} with  a constant $C'$ depending only on $\Cmodel$ and $\mathfrak{C}$.  

\subsection{Proof of Theorem \ref{thm: error 1}}

The proof of Theorem~\ref{thm: error 1} relies on a crucial approximation result, stated in Proposition~\ref{densityeuler} below, concerning the density of a diffusion process and that of its Euler scheme counterpart. The proof of Proposition~\ref{densityeuler} itself is based on the sharp results of Gobet and Labart \cite{Gobet2008sharp}.

\begin{prop}\label{densityeuler}
Work under Assumptions \ref{ass: AI}, \ref{ass: AII} and \ref{ass: AIII 1}.  Let $x \mapsto \bar \mu^h_t(x)$ be the density function of $\bar X^h_t$ defined by \eqref{eq: euler abstract}.  For any $0 < t_{\min} \leq T$, we have
\begin{equation}\label{deneuler}
\sup_{(t, x)\in[t_{\min}, T]\times \RD}\big| \bar \mu^h_t(x)-\mu_{t}(x)\big|\leq   C_{t_{\min}, \mathfrak{M}}\cdot  h,
\end{equation}
for some constant $C_{t_{\min}, \mathfrak{M}}$ depending only on $t_{\min}$ and $\Cmodel$. 
\end{prop}
\begin{proof}[Proof of Proposition \ref{densityeuler}]

For $x_0 \in \R^d$, let $(\xi_t^{x_0})_{t\in[0, T]}$ be a diffusion process of the form
\begin{equation}\label{diffusion}
\xi_t^{x_0}=x_0+\int_0^t\tilde{b}(s, \xi_s^{x_0})ds+\int_0^t\tilde{\sigma}(s, \xi_s^{x_0})dB_s,
\end{equation}
where $\widetilde \sigma: [0,T]\times \R^d \rightarrow \R^d \otimes \R^d$ and $\widetilde b:[0,T] \times \R^d \rightarrow \R^d$ are both $\mathcal C^{1,3}_b$ and $\widetilde \sigma$ is uniformly  elliptic  and $\partial_t \widetilde{\sigma}$ is  $\mathcal C_{b}^{0,1}$. We associate its companion
Euler scheme $(\bar{\xi}_{t}^{x_0,h})_{t\in[0, T]}$:
\begin{align*} \label{diffeuler}
\begin{cases}
\bar{\xi}_{t_{m+1}}^{x_0,h}=\bar \xi_{t_m}^{x_0,h}+h\cdot \tilde{b}(t_m, \bar{\xi}^{x_0, h}_{t_{m}})+\sqrt{h}\,\tilde{\sigma}(t_m, \bar{\xi}_{t_{m}}^{x_0,h})Z_{m+1}^n,\nonumber\\
Z_{m+1}^{n}\coloneqq \frac{1}{\sqrt{h}}\big( B_{t_{m+1}}- B_{t_{m}}\big),\noindent\\
\bar{\xi}_{t_{0}}^{x_0,h} = x_0,\\
\forall \, t\in[t_{m}, t_{m+1}), \quad \bar{\xi}_{t}^{x_0,h}=\bar{\xi}_{t_{m}}^{x_0,h}+(t-t_m) \tilde{b}(t_m, \bar{\xi}_{t_{m}}^{x_0,h})+\tilde{\sigma}(t_m, \bar{\xi}_{t_{m}}^{x_0,h})(B_t-B_{t_{m}}).
\end{cases}
\end{align*}
For $t>0$, both  $\mathcal L(\xi_t^{x_0})$ and $\mathcal L(\bar \xi_t^{x_0,h})$ are absolutely continuous, with density $p_t(x_0,x)$ and $\bar p_t^h(x_0,x)$ w.r.t. the Lebesgue measure $dx$ on $\R^d$. By 
Theorem 2.3 in \cite{Gobet2008sharp}, there exist two constants $c_1$ and $c_2$ depending on $T$, $d$ and the data $(\widetilde \sigma, \widetilde b)$ such that 
\begin{equation} \label{eq: gobetlabart}
\big|\bar{p}^h_{t}(x_0,x)-p_{t}(x_0,x)\big|\leq c_1 h \,t^{-\frac{d+1}{2}}\exp\Big(-\frac{c_2|x-x_0|^{2}}{t}\Big).
\end{equation}
Taking $\widetilde \sigma(t,x) = \sigma(t,x)$ and $\widetilde b(t,x) = b(t,x,\mu_t)$, we can identify  $\bar \mu^h_{t}(x)$  with $\int_{\R^d}\bar p_{t}^h(x_0,x)\mu_0(dx_0)$ and $\mu_t(x)$ with $\int_{\R^d}p_{t}(x_0,x)\mu_0(dx_0)$. 
For every $t \in [t_{\min}, T]$, we have 
\begin{align*}
\big|\bar \mu_t^h(x)-\mu_t(x)\big|  &  \leq \int_{\R^d}\big|\bar{p}_{t}^h(x_0,x)-p_t(x_0,x)\big|\mu_0(dx_0)\\
& \leq c_1 h\int_{\R^d} \,t^{-\frac{d+1}{2}}\exp\Big(-\frac{c_2|x-x_0|^{2}}{t}\Big)\mu_0(dx_0)\\
&\leq c_1t_{\min}^{-\frac{d+1}{2}}\cdot h =:  C_{t_{\min}, \mathfrak{M}}  \cdot h,
\end{align*}
by Theorem 2.3 in \cite{Gobet2008sharp} since Assumption \ref{ass: AIII 1} implies $(t,x) \mapsto b(t,x,\mu_t)$ is $\mathcal C_b^{1,3}$. More specifically,  
for $(X_t)_{t \in [0,T]}$ solving \eqref{mckeaneq}, we have
\[b(t,x,\mu_t) = \E\big[\,\widetilde b(t,x,X_t)\big].\]
By dominated convergence,  $x \mapsto b(t,x,\mu_t)$ is $\mathcal C^3_b$ thanks to the regularity assumptions on $x \mapsto \widetilde b(t,x,y)$.
For the regularity in time, by It\^o's formula
\begin{align*}
 \E\big[\,\widetilde b(t,\cdot,X_t)\big] & =  \E\big[\,\widetilde b(0,\cdot,X_0)\big] + \int_0^t \E\big[\partial_{1}\widetilde b(s,\cdot,X_s)\big]ds+\int_0^t\E\big[\mathcal A_s\widetilde b(s,\cdot,X_s)\big]ds,
\end{align*}
where $(\mathcal A_s)_{s \in [0,T]}$ is the family of generators defined in \eqref{eq: generator} applied to 
$$ y \mapsto \widetilde b(s,\cdot, y) = (\widetilde b(s,\cdot,y )_{[1]},\ldots, \widetilde b(s,\cdot,y)_{[d]})$$
componentwise. Moreover, for each component $1 \leq k \leq d$, 
\begin{align*}
\mathcal A_s\widetilde b(s,\cdot, X_s)_{[k]} & =  \E\big[\,\widetilde b(s,\cdot,X_s)]^{\top}\nabla_{y} \widetilde b(s,\cdot,X_s)_{[k]}+\tfrac{1}{2}\sum_{l,l'=1}^d c_{ll'}(s,X_s)\partial_{ y_ly_{l '}} \widetilde b(s,\cdot, X_s)_{[k]},
\end{align*}
hence
$$\E\big[\mathcal A_s\widetilde b(s,\cdot,X_s)_{[k]}\big] =\E\big[\,\widetilde b(s,\cdot,X_s)]^{\top}\E\big[F_k(s,X_s)\big]+\E\big[G_k(s,X_s)\big],$$
with
$$F_k(s,y) = \nabla_{y} \widetilde b(s,\cdot, y)_{[k]},\;\;\text{and}\;\;G_k(s,y) = \tfrac{1}{2}\sum_{l,l'=1}^d c_{ll'}(s,y)\partial_{y_ly_{l '}} \widetilde b(s,\cdot,y)_{[k]}.$$
All three functions $(s,y) \mapsto \widetilde b(s,\cdot,y)$, $(s,y)\mapsto F_k(s,y)$ and $(s,y) \mapsto G_k(s,y)$ are continuous and bounded by Assumption \ref{ass: AIII 1}, and so are 
$s \mapsto  \E\big[\,\widetilde b(s,\cdot,X_s)]$, $s \mapsto \E\big[F_k(s,X_s)\big]$ and $s \mapsto \E\big[G_k(s,X_s)\big]$ by dominated convergence using in particular that $s \mapsto X_s$ is stochastically continuous. Hence $s \mapsto \E\big[\mathcal A_s\widetilde b(s,\cdot,X_s)\big]$ is continuous and $s \mapsto \E\big[\partial_{s}\widetilde b(s,\cdot,X_s)\big]$ too, using again Assumption \ref{ass: AIII 1}. It follows that $t \mapsto b(t,x,\mu_t)$ is $\mathcal C_b^1$ on $[0,T]$. 
\end{proof}

\subsubsection*{Proof of Theorem \ref{thm: error 1}}

We have 
\begin{align*}
\E\big[\big|\widehat \mu^{N,h,\eta}_{T}(x)-\mu_{T}(x)\big|^p\big] & \leq 3^{p-1}\Big(\E\big[\big|\widehat \mu^{N,h,\eta}_{T}(x)-K_\eta \star \bar \mu^{h}_{T}(x)\big|^p\big]\\
&\hspace{3mm}+ \big|K_\eta \star \bar \mu^h_{T} (x)-K_\eta \star \mu_{T}(x)\big|^p
+\big|K_\eta \star \mu_{T}(x)-\mu_{T}(x)\big|^p\Big),
\end{align*}
and we plan to bound each term separately.
First, by Proposition \ref{prop: bernstein}
\begin{align*}
\mathbb E\big[\big|\widehat \mu^{N,h,\eta}_{T}(x)-K_\eta \star \bar \mu^{h}_{T}(x)\big|^p\big] & = \int_0^\infty \PP\big(\big|\widehat \mu^{N,h,\eta}_{T}(x)-K_\eta \star \bar \mu^{h}_{T}(x)\big| \geq z^{1/p}\big)\,dz \\
& \leq 2\kappa_1 \int_0^\infty \exp\Big(-\frac{\kappa_2Nz^{2/p}}{|K_\eta(x-\cdot)|^2_{L^2(\mu_{T})}+|K_\eta|_\infty z^{1/p}}\Big)\,dz \\
& \leq 2\kappa_1 \int_0^\infty \exp\Big(-\frac{\kappa_2N\eta^{d}z^{2/p}}{\sup_{y\in \mathrm{Supp}(K)}\mu_{T}(x-y)|K|^2_{L^2}+|K|_\infty z^{1/p}}\Big)\,dz \\
& \leq 2\kappa_1c_p \kappa_2^{-p/2}\sup_{y\in \mathrm{Supp}(K)}\mu_{T}(x-y)^{p/2}|K|^p_{L^2} (N\eta^{d})^{-p/2},
\end{align*}
stemming from the estimate
$$\int_0^\infty \exp\Big(-\frac{az^{2/p}}{b+cz^{1/p}}\Big)dz \leq c_p \max\Big(\Big(\frac{a}{b}\Big)^{-p/2}, \Big(\frac{a}{c}\Big)^{-p}\Big),$$
valid for $a,b,c,p >0$, with $c_p = 2\int_0^\infty \exp\big(-\tfrac{1}{2}(\min(z, \sqrt{z}))^{2/p}\big)dz$.
Next, 
\begin{align*}
\big|K_\eta \star \bar \mu^h_{T}(x)-K_\eta \star \mu_{T}(x)\big|^p & \leq \sup_{y \in \R^d}\big|\bar \mu_{T}^h(y)-\mu_{T}(y)\big|^p \big(\int_{\R^d}\big|K_\eta(x-y)\big|dy\big)^p \leq   \mathfrak{M}^p |K|_{L^1}^p h^p
\end{align*}
by Proposition \ref{densityeuler}. Finally, by definition 
$$\big|K_\eta \star \mu_{T}(x)-\mu_{T}(x)\big|^p \leq \mathcal B_{\eta}(\mu_{T}, x)^p,$$
and we obtain Theorem \ref{thm: error 1} with 
\begin{equation*}
\kappa_3 = 3^{p-1}\max\big(2\kappa_1c_p \kappa_2^{-p/2}\sup_{x\in\mathcal{K}, y\in \mathrm{Supp}(K)}\mu_{T}(x-y)^{p/2}|K|^p_{L^2}, \mathfrak{M}^p |K|_{L^1}^p, 1\big).\end{equation*}

\subsubsection*{Proof of Corollary \ref{cor: error}}
By Taylor's formula, we have, for $x,y \in \R^d$, $\eta >0$ and any $1 \leq k' \leq k$:
\begin{align}
\mu_{T}(x+\eta y)-\mu_{T}(x) & = \sum_{1 \leq |\alpha| \leq k'-1}\frac{\partial^\alpha \mu_{T}(x)}{\alpha !}(\eta y)^\alpha+k'\sum_{|\alpha|=k'}\frac{(\eta y)^\alpha}{\alpha !}\int_0^1(1-t)^{k'-1}
\partial^\alpha \mu_{T}(x+t\eta y)dt, \label{eq: taylor expa}
\end{align}
with multi-index notation $\alpha = (\alpha_1,\ldots, \alpha_d)$, $\alpha_i \in \{0,1,\ldots \}$, $\alpha ! = \alpha_1 ! \alpha_2 ! \cdots \alpha_d !$, $|\alpha| = \alpha_1+\alpha_2+\cdots +\alpha_d$, and for $x = (x_1,\ldots, x_d)$:
$$x^\alpha = x_1^{\alpha_1}  x_2^{\alpha_2}\cdots x_d^{\alpha_d},\;\;\text{and}\;\;\partial^\alpha f = \frac{\partial^{|\alpha|} f}{\partial^{\alpha_1} x_1\partial^{\alpha_2} x_2\cdots \partial^{\alpha_d} x_d}.$$
It follows that
\begin{align*}
& K_\eta \star \mu_{T}(x)-\mu_{T}(x) \\
& = \eta^{-d} \int_{\R^d}K\big(\eta^{-1}(x-y)\big)\big(\mu_{T}(y)-\mu_{T}(x)\big)dy \\
& = \int_{\R^d}K(-y)\big(\mu_{T}(x+\eta y)-\mu_{T}(x)\big)dy \\
& = \int_{\R^d}K(-y)\big(\min(k, \ell+1)\sum_{|\alpha|=\min(k,\ell+1)}\frac{(\eta y)^\alpha}{\alpha !}\int_0^1(1-t)^{\min(k, \ell+1)-1}
\partial^\alpha \mu_{T}(x+t\eta y)dt\big)dy,
\end{align*}
thanks to \eqref{eq: taylor expa} with $k' = \min(k,\ell+1)$ and the cancellation property \eqref{eq: cancellation} of the kernel $K$ that eliminates the polynomial term in the Taylor's expansion, hence
$$\big|K_\eta \star \mu_{T}(x)-\mu_{T}(x)\big| \leq  |\mu_{T}|_{\min(k, \ell+1),K} \eta^{\min(k, \ell+1)},$$
with 
\begin{equation} \label{eq: def semi norm}
|f|_{\min(k, \ell+1),K} = \sum_{|\alpha| = \min(k, \ell+1)} \frac{|\partial^\alpha f|_\infty}{\alpha !} \int_{\R^d}|y|^\alpha |K(y)|dy.
\end{equation} 
It follows that 
\begin{equation} \label{eq: bias multivariate}
\mathcal B_\eta(\mu_{T}, x)^p \leq |\mu_{T}|_{\min(k, \ell+1),K}^p \eta^{\min(k, \ell+1)p}.
\end{equation}
Applying now Theorem \ref{thm: error 1}, we obtain
\begin{align*}
\mathbb E\Big[\big|\widehat \mu_{T}^{N,h,\eta}(x)-\mu_{T}(x)\big|^p\Big] &  \leq \kappa_3 \big(|\mu_{T}|_{\min(k, \ell+1),K}^p \eta^{\min(k, \ell+1)p}+(N\eta^d)^{-p/2}+h^p\big) \\
& \leq \kappa_3 \big(|\mu_{T}|_{\min(k, \ell+1),K}^p+1)N^{-\min(k, \ell+1)p/(2\min(k, \ell+1)+1)}+h^p\big)
\end{align*}
with the choice $\eta = N^{-1/(2\min(k, \ell+1)+d)}$
hence the corollary with $\kappa_4 = \kappa_3(1+|\mu_{T}|_{\min(k, \ell+1),K}^p)$.

\subsection{Proof of Theorem \ref{thm: oracle}}

This is an adaptation of the classical Goldenshluger-Lepski method \cite{GL08, GL11, GL14}. We repeat the main arguments and highlight the differences due to the stochastic approximation and the presence of the Euler scheme. \\

Abbreviating $\mathsf A_{\eta}^{N}(T,x)$ by $\mathsf A_{\eta}^{N}$, we have, for every $\eta \in \mathcal H$,
\begin{align*}
&\mathbb E\Big[\big(\widehat \mu_T^{N,h, \widehat \eta^N(T,x) }(x)-\mu_{T}(x)\big)^2\Big] \\
& \leq 2\mathbb E\Big[\big(\widehat \mu_T^{N,h, \widehat \eta^N(T,x) }(x)-\widehat \mu^{N,h,\eta}_{T}(x)\big)^2\Big] + 2\mathbb E\big[\big(\widehat \mu^{N,h,\eta}_{T}(x)-\mu_{T}(x)\big)^2\big] \\
& \leq 2\mathbb E\Big[\Big\{\big(\widehat \mu_T^{N,h, \widehat \eta^N(T,x) }(x)-\widehat \mu^{N,h,\eta}_{T}(x)\big)^2-\mathsf V_{\eta}^N -\mathsf V^N_{\widehat \eta^N(T,x)} \Big\}_++ \mathsf V_{\eta}^N +\mathsf V^N_{\widehat \eta^N(T,x)}\Big] \\
&\hspace{2mm}+ 2\mathbb E\big[\big(\widehat \mu^{N,h,\eta}_{T}(x)-\mu_{T}(x)\big)^2\big] \\
&  \leq 2\mathbb E\Big[\mathsf{A}^N_{\max(\eta, \widehat \eta^N(T,x))}+ \mathsf V_{\eta}^N +\mathsf V^N_{\widehat \eta^N(T,x)}\Big] + 2\mathbb E\big[\big(\widehat \mu^{N,h,\eta}_{T}(x)-\mu_{T}(x)\big)^2\big] \\
&  \leq 2\big(\mathbb E\big[\mathsf{A}^N_{\eta}\big] +\mathsf V_{\eta}^N\big) +2\E\big[\mathsf{A}^N_{\widehat \eta^N(T,x)}+\mathsf V_{\widehat \eta^N(T,x)}^N\big] + 2\mathbb E\big[\big(\widehat \mu^{N,h,\eta}_{T}(x)-\mu_{T}(x)\big)^2\big] \\
& \leq 4\E\big[\mathsf{A}^N_{\eta}+\mathsf V_{\eta}^N\big] + \kappa_3(\mathcal B_\eta(x,\mu_{T})^2+(N\eta^d)^{-1}+h^2) \\
& \leq  \kappa_7 \big(\E\big[\mathsf{A}^N_{\eta}\big]+\mathsf V_{\eta}^N+\mathcal B_\eta(x,\mu_{T})^2+h^2\big)
\end{align*}
with  $ \kappa_7  = 2\max(2+\frac{\kappa_3}{\varpi |K|_{L^2}^2}, \kappa_3)$ 
by definition of $\widehat \eta^N(T,x)$ in \eqref{eq: def minimisation}, of $\mathsf V_{\eta}^N$ in \eqref{eq: def variance lepski} and using Theorem \ref{thm: error 1} with $p=2$ to bound the last term.\\

We next bound the term $\E\big[\mathsf{A}^N_{\eta}\big]$. For $\eta,\eta'\in \mathcal H$, with $\eta'\leq \eta$, we start with the decomposition
\begin{align*}
& \widehat \mu^{N,h,\eta}_{T}(x)-\widehat \mu^{N,h,\eta'}_{T}(x) \\
& = \big(\widehat \mu^{N,h,\eta}_{T}(x)-K_\eta \star \overline \mu^h_{T}(x)\big)+ \big(K_\eta \star \overline \mu^h_{T}(x)-K_\eta \star \mu_{T}(x)\big)+\big(K_\eta \star \mu_{T}(x)-\mu_{T}(x)\big)\\
& -\big(\widehat \mu^{N,h,\eta'}_{T}(x)-K_{\eta'} \star \overline \mu^h_{T}(x)\big)- \big(K_{\eta'} \star \overline \mu^h_{T}(x)-K_{\eta'} \star \mu_{T}(x)\big)-\big(K_{\eta'} \star \mu_{T}(x)-\mu_{T}(x)\big),
\end{align*}
and thus, by Proposition \ref{densityeuler},
we infer
\begin{align*}
\big|\widehat \mu^{N,h,\eta}_{T}(x)-\widehat \mu^{N,h,\eta'}_{T}(x)\big| &\leq \big|\widehat \mu^{N,h,\eta}_{T}(x)-K_\eta \star \overline \mu^h_{T}(x)\big|+\big|\widehat \mu^{N,h,\eta'}_{T}(x)-K_{\eta'} \star \overline \mu^h_{T}(x)\big| \\
&\hspace{2mm} +2  \Cmodel |K|_{L^1}h+2\,\mathcal B_{\eta}(\mu_{T},x)
\end{align*}
using $\eta' \leq \eta$ to bound the second bias term.
\begin{align*}
 \big(\widehat \mu^{N,h,\eta}_{T}(x)-\widehat \mu^{N,h,\eta'}_{T}(x)\big)^2- \mathsf{V}_\eta^N-\mathsf{V}_{\eta'}^N & \leq 3\big(\widehat \mu^{N,h,\eta}_{T}(x)-K_\eta \star \overline \mu^h_{T}(x)\big)^2-\mathsf{V}_{\eta}^N \\
&+3\big(\widehat \mu^{N,h,\eta'}_{T}(x)-K_{\eta'} \star \overline \mu^h_{T}(x)\big)^2-\mathsf{V}_{\eta'}^N\\
&+3\big(2  \Cmodel |K|_{L^1}h+2\mathcal B_\eta(x,\mu_{T})\big)^2.
\end{align*}
Taking maximum for $\eta' \leq \eta$, we thus obtain
\begin{align*}
\max_{\eta' \leq \eta}\big\{\big(\widehat \mu^{N,h,\eta}_{T}(x)-\widehat \mu^{N,h,\eta'}_{T}(x)\big)^2- \mathsf{V}_\eta^N-\mathsf{V}_{\eta'}^N \big\} & \leq   \big\{3(\widehat \mu^{N,h,\eta}_{T}(x)-K_\eta \star \mu_{T}(x))^2-\mathsf{V}_{\eta}^N\big\}_+ \\
&\hspace{3mm}+ \max_{\eta' \leq \eta} \big\{3(\widehat \mu^{N,h,\eta'}_{T}(x)-K_\eta \star \mu_{T}(x))^2-\mathsf{V}_{\eta'}^N\big\}_+\\
&\hspace{2mm}+24\big(  \Cmodel^2|K|_{L^1}^2h^2+\mathcal B_\eta(x,\mu_{T})^2\big).
\end{align*}
We finally bound the expectation of each term. First, we have, by Proposition \ref{prop: bernstein}
\begin{align*}
&\E\big[ \big\{3\big(\widehat \mu^{N,h,\eta}_{T}(x)-K_\eta \star \overline \mu^h_{T}(x)\big)^2-\mathsf{V}_{\eta}^N\big\}_+\big] \\
& = \int_0^\infty \PP\big(3\big(\widehat \mu^{N,h,\eta}_{T}(x)-K_\eta \star \overline \mu^h_{T}(x)\big)^2-\mathsf{V}_{\eta}^N\geq z\big)dz\\
& = \int_0^\infty \PP\big(\big|\widehat \mu^{N,h,\eta}_{T}(x)-K_\eta \star \overline \mu^h_{T}(x)\big| \geq 3^{-1/2}(z+\mathsf{V}_{\eta}^N)^{1/2}\big)dz\\
& \leq 2\kappa_1 \int_{V_\eta^N}^\infty \exp\Big(-\frac{\kappa_2N\eta^{d}\tfrac{1}{3}z}{\sup_{y\in \mathrm{Supp}(K)}\mu_{T}(x-y)|K|^2_{L^2}+|K|_\infty 3^{-1/2}z^{1/2}}\Big)\,dz \\
& \leq 2\kappa_1 \int_{V_\eta^N}^\infty \exp\Big(-\frac{\kappa_2N\eta^{d}z}{6\sup_{y\in \mathrm{Supp}(K)}\mu_{T}(x-y)|K|^2_{L^2}}\Big)\,dz+ 2\kappa_1 \int_{V_\eta^N}^\infty \exp\Big(-\frac{\kappa_2N\eta^{d}z^{1/2}}{2\sqrt{3}|K|_\infty}\Big)\,dz \\
& \leq \tfrac{12 \kappa_1}{\kappa_2}\sup_{y\in \mathrm{Supp}(K)}\mu_{T}(x-y)|K|^2_{L^2} (N\eta^d)^{-1}N^{-\kappa_2 \varpi/6\sup_{y\in \mathrm{Supp}(K)}\mu_{T}(x-y)}\\
&\hspace{2mm} + \tfrac{48\sqrt{3}\kappa_1}{\kappa_2} |K|_\infty |K|_{L^2} \varpi^{1/2}(N\eta^d)^{-3/2}(\log N)^{1/2}\exp\big(-\tfrac{\varpi^{1/2}\kappa_2|K|_{L^2}}{2\sqrt{3}|K|_\infty}(N\eta^d)^{1/2}(\log N)^{1/2}\big),
\end{align*}
where we used $\int_{\nu}^\infty \exp(-z^{1/2})dz \leq 4\nu^{1/2}\exp(-\nu^{1/2})$ for $\nu\geq 16$.  If \[\varpi \geq 12\kappa_2^{-1}\max(\sup_{x\in\mathcal{K},\,y \in \mathrm{Supp}(K)}\mu_{T}(x-y), 4|K|_{L^2}^{-2}|K|_\infty^2\kappa_2^{-1})\] and $\eta^d \geq N^{-1}\log N \tfrac{3 \log 2 |K|_\infty^2}{2 |K|_{L^2}^2}$, we have  
$$\E\big[ \big\{3\big(\widehat \mu^{N,h,\eta}_{T}(x)-K_\eta \star \overline \mu^h_{T}(x)\big)^2-\mathsf{V}_{\eta}^N\big\}_+\big] \leq  \kappa_8  (\log N)^{1/2}N^{-2},$$
with $ \kappa_8  = 2\max(\tfrac{12 \kappa_1}{\kappa_2}\sup_{x\in\mathcal{K},\,y\in \mathrm{Supp}(K)}\mu_{T}(x-y)|K|^2_{L^2} ,  \tfrac{48\sqrt{3}\kappa_1}{\kappa_2} |K|_\infty |K|_{L^2} \varpi^{1/2})$.

In the same way, we have the rough estimate
\begin{align*}
\E\big[ \max_{\eta' \leq \eta}\big\{3\big(\widehat \mu^{N,h,\eta}_{T}(x)-K_\eta \star \overline \mu^h_{T}(x)\big)^2-\mathsf{V}_{\eta}^N\big\}_+\big] & \leq \sum_{\eta' \in \mathcal H}\E\big[\big\{3\big(\widehat \mu^{N,h,\eta'}_{T}(x)-K_\eta \star \overline \mu^h_{T}(x)\big)^2-\mathsf{V}_{\eta'}^N\big\}_+\big] \\
&\leq  \mathrm{Card}(\mathcal H) \kappa_8 (\log N)^{1/2} N^{-2} \\
&\leq  \kappa_8 (\log N)^{1/2} N^{-1},
\end{align*}
using the previous bound. We thus have proved
\begin{align*}
\E\big[\mathsf{A}^N_{\eta}\big] & \leq 2 \kappa_8 (\log N)^{1/2}N^{-1}+24\big(\kappa_{10}^2|K|_{L^1}^2h^2+\mathcal B_\eta(x,\mu_{T})^2\big) \\
& \leq  \kappa_9 \big(\mathsf{V}_\eta^N+\mathcal B_\eta(x,\mu_{T})^2+h^2\big),
\end{align*}
with $ \kappa_9  = \max\big(2 \kappa_8 / |K|_{L^2}^2\varpi, 24\kappa_{10}^2|K|_{L^1}^2, 24\big)$. 
Back to the first step of the proof, we infer
\begin{align*}
 \mathbb E\Big[\big(\widehat \mu^{N,h, \widehat \eta^N(T,x) }_{T}(x)-\mu_{T}(x)\big)^2\Big] 
& \leq  \kappa_7 \big(\E\big[\mathsf{A}^N_{\eta}\big]+\mathsf V_{\eta}^N+\mathcal B_\eta(x,\mu_{T})^2+h^2\big) \\
& \leq \kappa_{5}\big(\mathsf V_{\eta}^N+\mathcal B_\eta(x,\mu_{T})^2+h^2\big),
\end{align*}
where now $\kappa_5 =  \kappa_7 (1+ \kappa_9 )$. Since $\eta \in \mathcal H$ is arbitrary, the proof of Theorem \ref{thm: oracle} is complete.
\subsubsection*{Proof of Corollary \ref{cor: oracle adaptive}}
For every $\eta \in \mathcal H$, we have
$$
\mathcal B_\eta(\mu_{T}, x)^2 \leq |\mu_{T}|_{k,K}^2 \eta^{2\min(k,\ell+1)}, 
$$
where $|\mu_{T}|_{k,K}$ is defined via \eqref{eq: def semi norm}, see \eqref{eq: bias multivariate} in the proof of Corollary \ref{cor: error}.  
By Theorem \ref{thm: oracle}, we thus obtain
\begin{align*}
 \mathbb E\Big[\big(\widehat \mu_{T}^{N,h, \widehat \eta^N(T,x) }(x)-\mu_{T}(x)\big)^2\Big] & \leq \kappa_{5}\max(|\mu_{T}|_{k,K}^2, \varpi |K|_{L^2}^2, 1)\\
& \hspace{2mm}\times \big(\min_{\eta \in \mathcal H}\big((N\eta^d)^{-1}\log N+ \eta^{2\min(k,\ell+1)}\big)+h^2\big) \\
& \leq  \kappa_6 \Big(\Big(\frac{\log N}{N}\Big)^{\frac{2\min(k,\ell+1)}{2\min(k,\ell+1)+d}}+h^2\Big)
\end{align*}
with $\kappa_6 =  2\kappa_{5}\max(|\mu_{T}|_{k,K}^2, \varpi |K|_{L^2}^2, 1)$, using the fact that the minimum in $\eta \in \mathcal H$ is attained for $\eta = (N/\log N)^{1/(2\min(k,\ell+1)+d)}$. 

\subsection{Proof of Theorem \ref{thm: nonlinear}}

We briefly outline the changes that are necessary in the previous proofs to extend the results to a nonlinear drift in the measure argument that satisfies Assumption \ref{ass: AIII 2}.

\subsubsection*{Theorem \ref{thm: deviation} under Assumption \ref{ass: AIII 2}} 
Only Lemma \ref{expp} needs to be extended to the case of a nonlinear drift in order to obtain Proposition \ref{prop: bernstein}, the rest of the proof remains unchanged. This amounts to prove that the random variables $b(\underline{s}, Y_{\underline{s}}^{n,h},\widetilde \mu^{N,h}_{\underline{s}})-b(\underline{s}, Y_{\underline{s}}^{n,h}, \bar \mu_{\underline{s}}^h) $ are sub-Gaussian, with the correct order in $N$. We may then repeat part of the proof of Proposition 19  of \cite{della2021nonparametric}. The smoothness assumption \eqref{eq: smoothness nonlinear b} for the nonlinear drift enables one to obtain
\begin{align*}
& \big|b(\underline{s}, Y_{\underline{s}}^{n,h},\widetilde \mu^{N,h}_{\underline{s}})-b(\underline{s}, Y_{\underline{s}}^{n,h}, \bar \mu_{\underline{s}}^h)\big|^2 \\
&\lesssim \sum_{l = 1}^{k-1}\big|\int_{(\R^d)^l}\delta_\mu^l b(s,Y_{\underline{s}}^{n,h}, y^l, \bar \mu_{\underline{s}}^h)(\widetilde \mu^{N,h}_{\underline{s}}-\bar \mu_{\underline{s}}^h)^{\otimes l}(dy^l)\big|^2+\min\big(\mathcal W_1(\widetilde \mu^{N,h}_{\underline{s}},\bar \mu_{\underline{s}})^2, \mathcal W_1(\widetilde \mu^{N,h}_{\underline{s}},\bar \mu_{\underline{s}})^{2k}\big),
\end{align*}
following the proof of Lemma 21 in \cite{della2021nonparametric}.
Next, $\min\big(\mathcal W_1(\widetilde \mu^{N,h}_{\underline{s}},\bar \mu_{\underline{s}})^2, \mathcal W_1(\widetilde \mu^{N,h}_{\underline{s}},\bar \mu_{\underline{s}})^{2k}\big)$ is sub-Gaussian with the right order in $N$, thanks to the sharp deviation estimates of Theorem 2 in Fournier and Guillin \cite{fournier2015rate}. As for the main part of the previous expansion, we use a bound of the form
$$\E\Big[\big|\int_{(\R^d)^l}\delta_\mu^l b(s,Y_{\underline{s}}^{n,h}, y^l, \bar \mu_{\underline{s}}^h)(\widetilde \mu^{N,h}_{\underline{s}}-\bar \mu_{\underline{s}}^h)^{\otimes l}(dy^l)\big|^{2p}\Big] \leq c^p N^{-p} p!\;\;\text{for every}\;\;p\geq 1,$$
for some constant $c$ depending on $b$, see Lemma 22 in \cite{della2021nonparametric}, showing eventually that this term is sub-Gaussian with the right order in $N$. Lemma \ref{expp} follows.

\subsubsection*{Theorem \ref{thm: error 1} under Assumption \ref{ass: AIII 2}} Again, only an extension of Proposition \ref{densityeuler} is needed, the rest of the proof remains unchanged. This only amounts to show that $(t,x) \mapsto b(t,x,\mu_t)$ is $\mathcal C_b^{1,3}$. The smoothness in $x$ is straightforward, and only the proof that $t \mapsto b(t,x,\mu_t)$ is $\mathcal C^1_b$ is required. By It\^o's formula, we formally have
$$\tfrac{d}{dt} b(t,\cdot,\mu_t) = \partial_1 b(t,\cdot,\mu_t)+\E\big[\mathcal A_t \delta_\mu b(t,\cdot,X_t,\mu_t)\big],$$
where $(X_t)_{t \in [0,T]}$ is a solution to \eqref{mckeaneq}. Assumption \ref{ass: AIII 2} enables one to conclude that $\tfrac{d}{dt} b(t,\cdot,\mu_t)$ is continuous and bounded.

\subsubsection*{Remaining proofs}
The proofs of Corollary \ref{cor: error}, Theorem \ref{thm: oracle} and Corollary \ref{cor: oracle adaptive} remain unchanged under Assumption \ref{ass: AIII 2}.

\bibliographystyle{alpha}
\bibliography{version_finale_arxiv}

\end{document}